\documentclass[preprint,12pt]{elsarticle}
 \usepackage[margin=1.2in]{geometry} 
\usepackage{amsmath,amsthm,amssymb,amsfonts, color, comment, graphicx, environ, mathrsfs,hyperref,blindtext,tikz,stackengine}
\usetikzlibrary{calc}
\pgfdeclarelayer{back}
\pgfdeclarelayer{fore}
\pgfsetlayers{back,main,fore}

\renewcommand{\theenumi}{\it (\roman{enumi})}

\usepackage[utf8]{inputenc}
\usepackage[english]{babel}

\newtheorem*{remark}{Remark}

\newtheorem{theorem}{Theorem}
\newtheorem{problem}[theorem]{Problem}

\newtheorem{corollary}[theorem]{Corollary}
\newtheorem{lemma}[theorem]{Lemma}
\newtheorem{claim}[theorem]{Claim}

\newtheorem{definition}[theorem]{Definition}

\newcommand{\E}{\mathbb{E}}
\newcommand{\prob}{\mathbb{P}}

\newcommand{\var}{\operatorname{Var}}
\newcommand{\tGSz}{\tilde{G}_{Sz}}

\newcommand{\GSz}{G_{Sz}}
\newcommand{\whp}{{\rm whp}}
\newcommand{\Whp}{{\rm Whp}}
\newcommand{\cG}{{\mathcal G}}
\newcommand{\cI}{{\mathcal I}}

\usepackage{amssymb}
\usepackage{amsmath}

\journal{Discrete Mathematics}

\begin{document}

\begin{frontmatter}


\author{Enrique Gomez-Leos\corref{cor1}}
\ead{enriqueg@iastate.edu}
\author{Ryan R. Martin}
\ead{rymartin@iastate.edu}
\address{Department of Mathematics, Iowa State University, Ames, IA, USA 50011}
\cortext[cor1]{Corresponding Author}

\title{Tiling randomly perturbed multipartite graphs}




\begin{abstract}
A perfect $K_r$-tiling in a graph $G$ is a collection of vertex-disjoint copies of the graph $K_r$ in $G$ that covers all vertices of $G$. In this paper, we prove that the threshold for the existence of a perfect $K_{r}$-tiling of a randomly perturbed balanced $r$-partite graph on $rn$ vertices is $n^{-2/r}$. This result is a multipartite analog of a theorem of Balogh, Treglown, and Wagner~\cite{balogh2019tilings} and extends our previous result, which was limited to the bipartite setting~\cite{gomezleos2024tilingrandomlyperturbedbipartite}.
\end{abstract}




\end{frontmatter}




\section{Introduction}
\label{s:intro}
In the study of extremal graph theory, many results concern the determination of a minimum degree condition that guarantees the existence of some spanning subgraph. For a fixed subgraph $H$, An \emph{$H$-tiling} of a graph $G$ is a subgraph consisting of vertex disjoint copies of $H$ and a \emph{perfect $H$-tiling} of $G$ is an $H$-tiling which spans all vertices of $G$. The celebrated result of Corr\'{a}di and Hajnal gives the minimum vertex degree necessary for finding a perfect $K_3$-tiling~\cite{corradi1963maximal}. Hajnal and Szemer\'{e}di generalized this result to cliques of arbitrary size~\cite{hajnal1970proof} and moreover showed that their result is best possible. Since then, there have been  generalizations to the multipartite setting, for instance~\cite{zhao2009bipartite, magyar2002tripartite, martin2008quadripartite, keevash2015multipartite}.

The Erd\H{o}s-R\'{e}nyi random graph $G(n,p)$ consists of the vertex set $[n]$ where each edge is present, independently, with probability $p=p(n)$. For the random graph $G(n,p)$, a key question is to establish the probability threshold for which $G(n,p)$ contains a fixed spanning subgraph.  The breakthrough result of Johansson, Kahn, and Vu~\cite{JKV} settled the threshold for which $G(n,p)$ admits a perfect $H$-tiling for a fixed \emph{strictly balanced} graph $H$, and in particular, the threshold for a perfect $K_{r}$-tiling, for any $r\geq 2$. Gerke and McDowell~\cite{GerkeMcDowell} determined the corresponding threshold for which $H$ is \emph{nonvertex-balanced} graph. 

In \cite{bohman2003many}, Bohman, Frieze, and Martin introduced the randomly perturbed graph model, which combines these two problems together. In the randomly perturbed setting, Balogh, Treglown, and Wagner~\cite{balogh2019tilings} determined the probability $p$ for the appearance of a perfect $H$-tiling in a graph on $n$ vertices with minimum degree at least $\alpha n$, for any graph $H$, and they showed that this is best possible for $\alpha<1/|V(H)|$~\cite[Section 2.1]{balogh2019tilings}. We state their result  for the case in which $H=K_r$ and $r\geq2$.

\begin{theorem}[Balogh, Treglown, Wagner~\cite{balogh2019tilings}, Theorem 1.3] \label{thm:btw}
    Let $r\geq 2$ and let $n \in \mathbb{N}$ be divisible by $r$. For every $\alpha >0$, there is a $c=c(\alpha, r)>0$ such that if $p\geq cn^{-2/r}$ and $G$ is an $n$-vertex graph with $\delta(G) \geq \alpha n$, then $G \cup G(n,p)$ contains a perfect $K_r$-tiling \whp.
\end{theorem}

In this paper, we consider tiling the \emph{randomly perturbed multipartite graph} which consists of two graphs on the same vertex set $V_1\sqcup \cdots \sqcup V_r$, $|V_1|=\cdots = |V_r| =n$. 
\begin{definition}
    Let $\alpha \in (0,1)$, $r\geq 2$, $n$ a positive integer. 
    A \emph{balanced} $r$-partite graph is one in which each vertex class has the same size.
    For a balanced $r$-partite graph $G = (V_1 \sqcup \cdots \sqcup V_r;E)$ on $rn$ vertices let
    $$\delta^{*}(G) :=  \min_{1\leq i,j<r}\bigl\{\delta(G[V_i, V_j])\bigr\}\geq \alpha n . $$
    Let $\mathcal{G}_{r}(\alpha;n)$ denote the set of all balanced $r$-partite graphs $G$ on $rn$ vertices with $\delta^{*}(G)\geq \alpha n$.
    Let $\mathcal{G}_{r}(\alpha)=\cup_n\mathcal{G}_r(\alpha;n)$.
    Let $G_r(n,p)$ denote the Erd\H{o}s-R\'enyi random graph on a balanced $r$-partite graph on $rn$ vertices such that edges between distinct vertex classes are present independently with probability $p$. 
\end{definition}

In this setting, one graph $G_n$ is an arbitrary member of $\cG_r(\alpha;n)$ and the other is a random graph $G_{r}(n,p)$, each of which is an $r$-partite graph, and each respects the same partition $(V_1, \ldots, V_r)$. Hence each of $V_1, \ldots, V_r$ will always be an independent set of vertices. More specifically, the notation $G_{r}(V_1, \ldots, V_r, p)$ refers to the random $r$-partite graph induced on the vertex sets $V_1,\ldots,V_r$.

Note that that we differ from the usual definition of the randomly perturbed graph in that we restrict the deterministic graph, $G_n\in\cG_r(\alpha;n)$, to be multipartite but we also restrict the appearance of the random edges to appear only between disjoint vertex classes.

We use $\prob(A)$ to denote the probability of event $A$ and use $\E[X]$ to denote the expectation of a random variable $X$. We say that a sequence of events $A_1,A_2,\ldots,A_n,\ldots$ occurs \emph{with high probability} (\whp) if $\lim_{n\to \infty} \prob (A_n) = 1$. We shall say that $G_{n}\cup G_{r}(n,p)$ has a graph property $P$ \emph{with high probability (\whp)} if 
\begin{align*}
    \lim_{n\to \infty} \prob \bigl(G_{n}\cup G_{r}(n,p) \in P\bigr) = 1.
\end{align*}
The question of interest is to determine the \emph{threshold} for which $G_{n}\cup G_{r}(n,p)$ admits a spanning subgraph.
 Keevash and Mycroft~\cite{keevash2015multipartite} (Theorem~\ref{thm:KM} below) proved that if $\alpha \geq (1-1/r) + 1/n$, then $p=0$ is sufficient, that is, no random edges are needed. If the host graph is initially empty, that is $\alpha =0$, Gerke and McDowell~\cite{GerkeMcDowell} showed that with $p =\Omega\Bigl( n^{-2/r}\log^{1/\binom{r}{2}}n\Bigr)$, $G_{r}(n,p)$ contains a perfect $K_r$-tiling \whp.


 In the context of this paper, we define the threshold function as follows:
\begin{definition}\label{def:threshold}
Given a monotone property $P$ and a class of balanced $r$-partite graphs $\cG_r$, the function $t(n): \mathbb{N} \to (0,1)$ is a \emph{threshold function of $\cG$ for $P$} if there exist real positive constants $c, C$ such that 
\begin{enumerate}
    \item If $p(n) \geq Ct(n)$, then for any sequence  $(G_{n})\subseteq \cG_r$, the graph $G_{n}\cup G_{r}(n,p)$ has property $P$, \whp.\label{def:th1}
    \item If $p(n) \leq ct(n)$, then there exists a sequence  $(G_{n}^{*}) \subseteq \mathcal{G}_{r}$, such that the graph $G_{n}^{*}\cup G_{r}(n,p)$ does not have property $P$, \whp.\label{def:th2}
\end{enumerate}
\end{definition} 

Inspired by Theorem~\ref{thm:KM}, along with other multipartite results, we prove Theorem~\ref{thm: main}, which is a multipartite version of Theorem~\ref{thm:btw}.
\begin{theorem}[Keevash and Mycroft~\cite{keevash2015multipartite}]\label{thm:KM}
    If $n$ is sufficiently large and $G_n\in\cG_r(\frac{r-1}{r};n)$, then $G$ has a perfect $K_r$-tiling unless both $r$ and $n$ are odd and $n$ is divisible by $r$. In that case, there is a single counterexample for which $\delta^*$ is exactly $\frac{r-1}{r}n$. 
\end{theorem}
\begin{theorem}\label{thm: main}
    For all $\alpha\in (0,1/r)$, $r\geq 3$,  
    the threshold of $\cG_r(\alpha)$ for the property of having a perfect $K_r$-tiling is $n^{-2/r}$. 
\end{theorem}

The statement of Theorem~\ref{thm: main} is proved in two parts. 
The first part, in Section~\ref{s:extremal example}, is a sequence of examples of $r$-partite graphs $(G_n^*)\subseteq\cG_r(\alpha)$ and a positive real constant $c>0$ such that $G_{n}^{*}\cup G_{r}(n,p)$ admits no perfect $K_r$-tiling \whp~if $p \leq cn^{-2/r}$. 
For the second part, in Section~\ref{s:proof of main}, we show that there exists a real constant $C>0$ such that any sequence $(G_{n})\subseteq\cG_r(\alpha)$, $G_{n}\cup G_{r}(n,p)$ admits a perfect $K_r$-tiling \whp~if $p \geq C n^{-2/r}$. 
In the event that no minimum degree is assumed (i.e. $\alpha =0$), Gerke and McDowell~\cite{GerkeMcDowell} established that a multiplicative polylog factor is required. 
In fact, similar to a result of Chang et al.~\cite{hypergraphfactors2022}, in Section~\ref{section:note on min degree}, we prove that a multiplicative $\omega(1)$ factor is required whenever $\alpha =o(1)$. 

This phenomenon of the linear minimum degree removing a multiplicative polylog factor appears in the tiling problem for the general randomly perturbed graph case \cite{balogh2019tilings, HMT, bottcher2023triangles}, as well as in the problem of finding Hamiltonian cycles \cite{bohman2003many}, spanning trees \cite{krivelevich2017bounded}. This phenomenon also occurs in the hypergraph setting~\cite{krivelevich2016cycles, hypergraphfactors2022}.

Our proof of Theorem~\ref{thm: main} follows many of the same standard arguments found both in the bipartite setting~\cite{gomezleos2024tilingrandomlyperturbedbipartite} as well as in the general setting~\cite{balogh2019tilings}. However, the first issue one faces when considering say, a $K_3$-tiling, in the case of $r=3$ of Theorem~\ref{thm: main}, is determining a suitable spanning subgraph by which to tile the Szemer\'{e}di graph. 
Indeed, the authors of \cite{balogh2019tilings} make use of a result of K\'{o}mlos~\cite{komlostilingturan} to tile with appropriately-sized stars. 
In the bipartite case~\cite{gomezleos2024tilingrandomlyperturbedbipartite}, there is a bipartite analog of K\'{o}mlos' result, attributed to Bush and Zhao~\cite{zhao2009bipartite}, that is used to tile the Szemer\'{e}di graph by disjoint stars. 

Unfortunately, in even the tripartite setting, the most naive approach requires a partial star tiling in which there is a roughly equal number of stars centered in each part of the tripartition. To our knowledge, no such theorem can be found in the literature. 
Instead, we rely on the linear programming method used by Martin, Mycroft, and Skokan~\cite{martin2017asymptotic}. Essentially, this method allows us to obtain a fractional star tiling that approximates the aforementioned partial star-tiling, we then employ standard Regularity Lemma arguments to obtain our desired tiling of Szemer\'{e}di graph. 
\subsection{Organization}
In Section~\ref{sec:prelim} we provide preliminaries. In Section~\ref{s:extremal example}, we provide an example which verifies Definition~\ref{def:threshold}\ref{def:th2} for Theorem~\ref{thm: main}. In Section~\ref{s:proof of main}, we verify Definition~\ref{def:threshold}\ref{def:th1} for Theorem~\ref{thm: main}. In Section~\ref{s:proofs of aux lemmas}, we prove some auxiliary lemmas used in Section~\ref{s:proof of main}. In Section~\ref{section:note on min degree}, we show that the linear minimum degree term in Theorem~\ref{thm: main} cannot be replaced by a sublinear term. In Section~\ref{section:conclusion}, we give some concluding remarks and state some future directions.

\section{Preliminaries}\label{sec:prelim}

\subsection{Notation}
For a graph $H$, we will use the notation $v_H = |V(H)|$ and $e_H = |E(H)|$ and $\chi(H)$ for the chromatic number of $H$. Let $\delta(H)$ denote the minimum degree of $H$, that is, the minimum size of the neighborhood of any vertex in $H$. Given a graph $G$, we use $N(v)$ to denote the neighborhood of a vertex $v\in V(G)$. We use $\deg_G(v)$ (or $\deg(v)$ when the context is clear) to denote $|N(v)|$ and we use $\deg(v,A)$ to denote $|N(v) \cap A|$ where $A \subseteq V(G)$. As is typical, we will ignore floors and ceilings when it does not matter.

\subsection{Concentration inequalities}
We begin with a version of the well-known Chebyshev Inequality. 

\begin{lemma}[Chebyshev Inequality]\label{lem:chebyshev} 
If $X$ is a random variable with mean $\E[X]$ and variance $\var(X)$, then 
    \begin{align*}
        \prob\bigl(|X-\E[X]| > \E[X]/2\bigr) \leq 4\var(X)/(\E[X])^2.
    \end{align*}
\end{lemma}

The version of the Chernoff bound that we use can be found in~\cite{JLR}. Specifically, Corollaries 2.3 and 2.4 and Theorem 2.10. 
\begin{lemma}[Chernoff Bounds, see~\cite{JLR}, Section 2.1]\label{lemma: chernoff}
    Let $X$ be either a binomial or hypergeometric random variable. Let $\xi\in (0,1)$. Then,
    \begin{align*} 
        \prob\bigl(|X - \E[X]| \geq \xi \E [X]\bigr) &\leq 2 \exp\bigg(-\frac{\xi^{3}}{3} \E [X]\biggr). 
    \end{align*}
    Moreover, for any $k\geq 7\E[X]$, we have $\prob(X >k) \leq \exp\{-k\}$.
\end{lemma}

At several points we will use a concentration result due to Janson~\cite{poisson} which leverages dependencies among random variables. 
The setting is that there is a ground set of elements $[N]$ and subsets $\{D_i\subset [N] : i\in\cI\}$. 
A subset $R\subseteq [N]$ is chosen such that each element $s\in [N]$ is a member of $R$ independently with probability $q_s\in (0,1)$. 
Let $I_i$ be the indicator of the event that $D_i\subset R$. 
We write $i\sim j$ if $D_i\cap D_j\neq\emptyset$. Note that $i\sim i$ for all $i\in\cI$. Lemma~\ref{lem:janson} is as follows:

\begin{lemma}[Janson's Inequality, see \cite{randomgraphsFK}, Theorem 21.12] \label{lem:janson}
    With the set up as above, let $S = \sum_{i\in\cI} I_i$ and let 
    \begin{align*}
        \Delta = \sum_{(i,j):i\sim j,i\neq j}\E[I_i I_j] ,
    \end{align*}
    where the summation is on ordered pairs. Let $0\leq t\leq\E[S]=:\mu$ and let $\varphi(x)= (1+x)\ln(1+x)-x$, then
    \begin{align*}
        \prob\bigl(S \leq \mu-t\bigr) \leq
        \exp\biggl\{-\frac{\phi(-t/\mu)\mu^2}{\Delta + \mu}\biggr\}
        \leq \exp\biggl\{-\frac{t^2}{2\bigl( \Delta + \mu \bigr)} \biggr\}.
    \end{align*}
\end{lemma} 

For the following corollary, the upper bounds come from setting $t=\mu$. 
The lower bound comes from Janson~\cite{poisson}.
\begin{corollary}\label{cor:janson}
    With the set up as in Lemma~\ref{lem:janson}, $\mu'=\sum_{i\in\cI}-\ln\bigl(1-\E[I_i]\bigr)$, and $\E[I_i]\leq 1/2$ for all $i\in\cI$, 
    $$ \exp\bigl\{-\mu'\bigr\}\leq \prob(S=0) \leq \exp\biggl\{-\frac{\mu^2}{\Delta+\mu}\biggr\} \leq \exp\bigl\{-\mu+\Delta\bigr\}. $$
\end{corollary}

\begin{remark}
    The setting for Lemma~\ref{lem:janson} and for Corollary~\ref{cor:janson} is similar to that of the Lov\'asz Local Lemma. However, the lower bound given in Corollary~\ref{cor:janson} is much stronger than that given by the Local Lemma and exploits the small amount of pairwise codependency that exists among the variables $\{I_i\}_{i\in\cI}$.
\end{remark}

\subsection{Epsilon-regular pairs}
For disjoint vertex sets $A$ and $B$, let $e(A,B)$ denote the number of edges with an endpoint in $A$ and an endpoint in $B$. Some definitions in papers using Szemer\'edi's Regularity Lemma vary slightly, we will follow the definitions in~\cite{balogh2019tilings}.

\begin{definition}
For disjoint vertex sets $A$ and $B$, the \emph{density} between $A$ and $B$ is
\begin{align*}
    d_G(A,B):= \frac{e(A,B)}{|A||B|}.
\end{align*}

Given $\epsilon>0$, we say that a pair of disjoint vertex sets $(A,B)$ is \emph{$\epsilon$-regular} if for all sets $X \subseteq A$, $Y\subseteq B$ with $|X| \geq \epsilon |A|$ and $|Y| \geq \epsilon |B|$ we have 
\begin{align*}
    |d_{G}(A,B) - d_{G}(X,Y) | < \epsilon.
\end{align*}
Given $d \in [0,1]$, we say that a pair of disjoint vertex sets $(A,B)$ is \emph{$(\epsilon,d)$-super-regular}\label{def: super regular} if the following two properties hold:
\begin{enumerate}
    \item for all sets $X \subseteq A$, $Y\subseteq B$ such that $|X| \geq \epsilon |A|$ and $|Y| \geq \epsilon |B|$, we have $d_{G}(X,Y) >d$;\label{def: sr density}
    \item for all $a \in A$ and $b \in B$, we have $\deg_{G}(a) > d|B|$ and $\deg_{G}(b) >d|A|$.\label{def: sr degree}
\end{enumerate}
\end{definition}

\begin{remark}
Note that condition~\ref{def: sr density} for super regularity is weaker than that for regularity. 
In this paper, it will be the case that in the work that follows that whenever a pair $(A,B)$ is $(\epsilon,d)$-super-regular for some $d$, that pair is also $\epsilon$-regular. 
\end{remark}

Later we will use random slicing with respect to our Szemer\'edi partition. The following technical lemma, Lemma~\ref{lemma: random slicing min deg} establishes that \whp~for all vertices, the proportion of neighbors in a set does not change by much if one chooses a random subset. This is a common argument, found in e.g. Balogh, Treglown, and Wagner~\cite{balogh2019tilings}. The proof is a standard probabilistic argument and is therefore omitted.
\begin{lemma}\label{lemma: random slicing min deg}
Let $0< \alpha, \beta \leq 1$ and $0<\beta'\leq\beta$. Given a bipartite graph $G = (A,B;E)$, where $|A|=\alpha n$ and $|B|=\beta n$ if $B' \subseteq B$ is chosen uniformly at random from all sets of size $\beta'n$, then for every $a\in A$, $\E[\deg(a,B')]=\deg(a,B)\frac{\beta'}{\beta}$. 
Furthermore, \whp~it is the case that for each $a\in A$, we have $\deg(a,B') \geq \deg(a,B)\frac{\beta'}{\beta}-\sqrt{2n}\,\ln n$.
\end{lemma}

We will use several properties of $\epsilon$-regular-pairs. The proofs can be found, e.g. as Lemma 1.4 in \cite{treglown2007regularity} and are omitted here. 
\begin{lemma}\label{lem: reg pairs not too many low degree]}
    Let $(A,B)$ be an $\epsilon$-regular pair with density $d=d(A,B) > \epsilon>0$. Then for every $Y\subseteq B$ with $|Y| \geq \epsilon |B|$, the number of vertices from $A$ with degree into $Y$ less than $(d-\epsilon)|Y|$ is at most $\epsilon|A|$.
\end{lemma}

The following, Lemma~\ref{lemma: regular matching}, is a consequence of Hall's Theorem.

\begin{lemma}\label{lemma: regular matching}
    For any $d>0$, there exists $\epsilon >0$ such that any $(\epsilon,d)$-super-regular pair $(A,B)$ with $|A| = |B|$ contains a perfect matching.
\end{lemma}


We will make use of the (deterministic) Slicing Lemma~\ref{lem: super-regular slicing}, which states that super-regular pairs contains subsets that are again super-regular with relaxed parameters.
\begin{lemma}\label{lem: super-regular slicing}
Suppose $(A,B)$ is an $(\epsilon,d)$-super-regular pair. If at most $\epsilon_1$ vertices are removed from each of $A$ and $B$ to obtain $A' \subseteq A$ and $B'\subseteq B$, then $(A',B')$ is $(\epsilon',d-\epsilon_1)$-super-regular with $\epsilon' = \max \{ \epsilon/\epsilon_1, 2\epsilon \}$.
\end{lemma}

We will also make use of the Random Slicing Lemma~\ref{lemma: random slicing martin-skoken}, which states that we can obtain super-regular pairs via random partitioning of exisiting super-regular pairs.
\begin{lemma}[Random Slicing, see \cite{martinskoken2017}, Lemma 10]\label{lemma: random slicing martin-skoken}
Let $0<d<1$, $0<\epsilon< \min\{d/4, (1-d)/4,1/9  \}$ and $D$ be a positive integer. There exists a $C_{\ref{lemma: random slicing martin-skoken}}= C_{\ref{lemma: random slicing martin-skoken}}(\epsilon,d)>0$ such that the following holds: Let $(X,Y)$ an $\epsilon$-regular pair with density $d$ with $|X|=|Y|=DL$. Let $X$ and $Y$ are partitioned into sets
$A_1 \cup A_2 \cup \dots \cup A_D$ and $B = B_1 \cup B_2 \cup \dots \cup B_D$ respectively, with $|A_i| = |B_i|=L$ for all $i$. Then with probability at least $1-\exp\{-C_{\ref{lemma: random slicing martin-skoken}}DL\}$, all pairs $(A_i, B_j)$ are $(16\epsilon)^{1/5}$-regular with density at least $d-\epsilon$.
\end{lemma}

\subsection{Szemer\'edi's Regularity Lemma}
Finally, we state a multipartite version of the degree form of Szemer\'edi's Regularity Lemma which can be derived from the original. See \cite[Theorem 2.8]{martin2017asymptotic} for the statement of the degree form. We will refer to this as ``the Regularity Lemma'' throughout this paper.

\begin{lemma}[Szemer\'edi's Regularity Lemma, multipartite degree form]\label{lemma: multipartite regularity}
    For every integer $r \geq 2$ and every $\epsilon>0$, there is an $M = M(r,\epsilon)$ such that if $G = (V_1, V_2, \dots, V_r; E)$ is a balanced $r$-partite graph on $rn$ vertices and $d \in [0,1]$ is any real number, then there exists integers $\ell$ and $L$, a spanning subgraph $G' = (V_1, \dots, V_r;E')$ and for each $i=1,\dots,r$ a partition of $V_i$ into clusters $V_{i}^{0}, V_{i}^{1}, \dots, V_{i}^{\ell}$ with the following properties:
    \begin{enumerate}
        \item $\ell \leq M$,
        \item $|V_i^{0}| \leq \epsilon n$ for all $i \in [r]$,
        \item $|V_{i}^{j}| = L \leq \epsilon n$ for $i\in [r]$ and $j \in [\ell]$,
        \item $\deg_{G'}(v,V_{i'}) > \deg_{G}(v,V_{i'}) - (d+\epsilon)n$ for all $v \in V_i$, $i \not = i'$\label{bp: degrees} and 
        \item all pairs $(V_{i}^{j}, V_{i'}^{j'})$ with $i \not = i'$, $j,j' \in [\ell]$ are $\epsilon$-regular with density exceeding $d$ or $0$.    
    \end{enumerate}
\end{lemma}

After applying Szemer\'{e}di's Regularity Lemma (Lemma~\ref{lemma: multipartite regularity}) to the deterministic graph $G$, we will define the \emph{Szemer\'{e}di graph} $\GSz$ obtained by taking its vertices as the vertex classes $V_i$ of $G$ with edges $\bigl\{V_i,V_j\bigr\}$ whenever $\bigl(V_i, V_j\bigr)$ forms an $\epsilon$-regular pair with density at least $d$. The Szemer\'{e}di graph partially inherits the minimum degree of $G$. Lemma~\ref{lemma: reduced graph degree} makes this statement precise. 

\begin{lemma}\label{lemma: reduced graph degree}
Let $\epsilon \ll d \ll \alpha$ and $r$ be a positive integer. If $G\in \cG_{r}(\alpha;n)$, then its Szem\'{e}redi graph $\GSz:= \GSz(\epsilon,d)$ on $r\ell$ vertices, has $\delta^{*}(\GSz) \geq \bigl(\alpha - \frac{d}{r} - (1+\frac{2}{r}\bigr) \epsilon) \ell$. 
\end{lemma}

Given a bounded degree subgraph $J$ of the Szemer\'{e}di graph, we can remove a few vertices from each cluster so that for the resulting graph, every pair that was regular is still regular with a relaxed parameter and every pair in $E(J)$ itself is super-regular.

\begin{lemma}\label{lemma: super-regularization}
    Let $0<d\ll 1$ and $\Delta$ and $\epsilon$ be such that $\Delta \cdot \epsilon<(d-\epsilon)/2$. There is a $\delta_{\ref{lemma: super-regularization}}$ and an $L_{\ref{lemma: super-regularization}}$ such that for all $L\geq L_{\ref{lemma: super-regularization}}$, the following holds\footnote{In fact, any $\delta$ satisfying $\delta<\Delta\epsilon<\delta+\delta^2$ suffices, for instance $\delta = 2\Delta\epsilon/(2+\Delta\epsilon)$ because $\Delta\epsilon<(d-\epsilon)/2<1/2$.}: 
    
    Let $\GSz$ be a Szemer\'edi graph with clusters of size $L$ such that every pair is $\epsilon$-regular with density at least $d$ for pairs in $E(\GSz)$ and with density zero for pairs not in $E(\GSz)$. Let $J$ be a subgraph of $\GSz$ with maximum degree at most $\Delta$.
    Let $L'\geq(1-\delta)L$ be an integer. 
    For every $A\in V(\GSz)$ there is a $A'\subset A$ of size exactly $L'$ such that
    \begin{enumerate}
        \item For every $\bigl(A,B\bigr)\in E\bigl(\GSz\bigr)$, the pair $\bigl(A',B'\bigr)$ is $2\epsilon$-regular with density at least $d-\epsilon$. \label{it:super-reg:reg}
        \item For every $\bigl(A,B\bigr)\in E\bigl(J\bigr)$, the pair $\bigl(A',B'\bigr)$ is $\bigl(2\epsilon,\delta\bigr)$-super-regular.\label{it:super-reg:super}
        \end{enumerate}
    \label{lemma:super-regularization}
\end{lemma}

The proof follows a standard argument in which a small set of vertices is deleted and it contains every vertex having small degree into an adjacent cluster.

\subsection{Random multipartite graphs}
Theorem~\ref{thm: covering} implies that the polylog factor in the threshold for a perfect bipartite tiling is necessary as it is in the general random graph case in Johansson, Kahn, and Vu~\cite{JKV}. However, in Theorem~\ref{thm: main} we prove that, in the perturbed case the threshold does not have this polylog factor. To that end, we state a special case.

\begin{theorem}[Gerke and McDowell~\cite{GerkeMcDowell}, Theorem 1.2]\label{thm: covering}
    For $r\geq 2$, 
   \[\lim_{n\to\infty } \mathbb{P}(G_{r}(n,p) \text{ contains a perfect $K_{r}$-tiling}) =\begin{cases} 
      1, & \text{if } p =\omega\bigl((\log n)^{2/(r(r-1))}n^{-2/r}\bigr); \\
      0,  & \text{if } p =o\bigl((\log n)^{2/(r(r-1))}n^{-2/r}\bigr).
   \end{cases}
\]
\end{theorem}

Lemma~\ref{lemma: partial covering} below will be useful for finding many copies of $K_{r}$ in sufficiently large, dense subgraphs. It is proved along the same lines as \cite[Theorem~4.9]{JLR} and was proved originally by Ruci\'nski~\cite{RucinskiExtension} in a more general setting in which the graph to be tiled need not be $K_{r}$ but can be any ``strictly balanced" graph and $G_{r}(n,p)$ is replaced by $G(n,p)$.

\begin{lemma}[Partial $K_r$-tiling]\label{lemma: partial covering}
Let $\epsilon \in (0,1/2)$, and $r\geq 3$ be a positive integer. Let $F(\epsilon,r)$ be the property that $G_{r}(n,p)$ contains a $K_{r}$-tiling that covers all but at most $\epsilon n$ vertices in each class. There exist $C_{\ref{lemma: partial covering}}=C_{\ref{lemma: partial covering}}(\epsilon,r)$ and $c_{\ref{lemma: partial covering}}=c_{\ref{lemma: partial covering}}(\epsilon,r)$ such that
\[\lim_{n\to\infty } \mathbb{P}\bigl(G_{r}(n,p) \in F(\epsilon,r)\bigr) =\begin{cases} 
    1,  & \text{if }p \geq C_{\ref{lemma: partial covering}} n^{-2/r};\\
      0& \text{if } p \leq c_{\ref{lemma: partial covering}}n^{-2/r}.
   \end{cases}
\]
\begin{proof}
Let $X$ be the number of copies of $K_r$ in $G_{r}(n,p)$ so that $\E[X]= n^{r}p^{{r \choose 2}}$. 

If $p < n^{-2/(r-1)}$, then $\E[X] < 1$ and by Markov's inequality, $\prob(X\geq (1-\epsilon)n) \leq \E[X]/(1-\epsilon)n=o(1)$, hence \whp~there are not enough copies of $K_r$ in $G_{r}(n,p)$.

If $n^{-2/(r-1)}\leq p\leq cn^{-2/r}$, then 
\begin{align*}
    \var(X)\leq \E[X^2]=n^{r} \sum_{\ell=2}^{r-1}n^{r-\ell} \binom{r}{\ell}p^{2\binom{r}{2} - \binom{\ell}{2}}\leq n^{2r} p^{2\binom{r}{2}} r 2^r n^{-1} = r 2^r \cdot\E[X]^2 n^{-1}.
\end{align*}
By Chebyshev's inequality (Lemma~\ref{lem:chebyshev}),
\begin{align*}
    \prob\bigl(X > 3\E[X]/2 \bigr)< 4\var(X)/(\E[X])^2 = O(1/n)=o(1).
\end{align*}

Therefore, \whp, $X <  \frac{3}{2}\E[X] < \frac{3}{2}c^{r \choose 2}n\leq (1-\epsilon )n$ for $c \ll 1/2$. 
Again, there are not enough copies of $K_r$ in $G_{r}(n,p)$.

Now, assume that $p\geq Cn^{-2/r}$. Suppose that $G_{r}(n,p) \not \in F(\epsilon,r)$; that is, there exist at least $\epsilon n$ vertices from each of the $r$ parts not containing a copy of $K_r$. In order to arrive to a contradiction, we will bound the probability that a copy of $K_r$ is not contained in $G_{r}(\epsilon n,p)$. 



We will use the corollary of Janson's inequality, Corollary~\ref{cor:janson}. Let $X$ be the number of copies of $K_r$ in $G_{r}( \epsilon n,p)$, so $\E[X] = (\epsilon n)^{r}p^{\binom{r}{2}}$. Let $I_i$ be the indicator variable for the event that the $i^{\rm th}$ copy of $K_r$ appears in $G_{r}( \epsilon n,p)$. For ease of notation let $m= \epsilon n$.
\begin{align*}
    \Delta &= \sum_{(i,j): i \sim j,i\neq j} \E[I_i I_j] 
    = m^r \sum_{\ell =2}^{r-1} \binom{r-1}{\ell} (m-1)^{r-\ell}p^{2{r\choose 2} - {\ell \choose 2}} \\
     &\leq \bigl(\E[X]\bigr)^2 \sum_{\ell =2}^{r-1} \binom{r-1}{\ell} m^{-\ell}p^{- {\ell \choose 2}} 
     =O_{\epsilon,r}\bigl(n^{-2+\frac{2}{r}}\bigr)\bigl(\E[X]\bigr)^2 ,
\end{align*}
as long as $C\geq 1$. In fact, for $C$ sufficiently large,
\begin{align*}
    \frac{\E[X]+\Delta}{(\E[X])^2} & 
    \leq \frac{1}{\E[X]} + O_{\epsilon,r}\bigl(n^{-2+\frac{2}{r}}\bigr)
    \leq r^{-1}n^{-1}.
\end{align*}

So by Corollary~\ref{cor:janson}, and the union bound, the probability that there exists sets of size $\epsilon n$ not containing a copy of $K_r$ is at most
\begin{align*}
    \binom{n}{ \epsilon n }^r \exp\biggl\{- \frac{(\E[X])^2}{\Delta +\E[X]} \biggr\}\leq  \exp\bigl\{rn\ln 2-r n\bigr\}=o(1).
\end{align*}
\end{proof}

\end{lemma}

\subsection{Linear programming}\label{ss:lp prelim}
In the proof of Theorem~\ref{thm: main}, we  will make use of the linear programming method as seen in \cite{martin2017asymptotic} and \cite{martinskoken2017}. We first provide the necessary background and follow the notation used by Martin, Mycroft, and Skokan~\cite{martinskoken2017}. 

A \emph{labeled graph} $H$ is a graph $H$ with an assignment $\lambda_{H}:V(H) \to \mathbb{R}_{\geq 0}$ to the vertices of $H$.

Denote by $\mathcal{K}_H(G)$ the set of subgraphs in $G$ isomorphic to 
a labeled $H$. A 
fractional $H$-tiling in $G$ is a weight assignment $w(H') \geq 0$ to each $\mathcal{K}_H(G)$ such that
\begin{align}
    \sum_{H' \in \mathcal{K}_H(G): v \in H' } w(H')\cdot\lambda_{H'}(v) \leq  1, \text{ for all }v\in V(G).\label{ineq:fractional} 
\end{align}
A fractional $H$-tiling is \emph{perfect} if we have equality in \eqref{ineq:fractional} for every $v \in V (G)$.

Let $r\geq3$ be a positive integer and $t$ be a positive rational number. 
Fix a balanced $r$-partite graph $V(G) = V_1 \sqcup V_2 \sqcup \cdots \sqcup V_r$ on $rn$ vertices. The subgraphs we will be concerned with are labeled copies of $K_{1,r-1}$, denoted $S^*$ as follows.

The set $S_t^{*}(i,j)$ consists of all labeled copies of $K_{1,r-1}$ in $G$ for which the \emph{center} vertex is in $V_i$, a special designated leaf called the \emph{big leaf} is in $V_j$, and the remaining $r-2$ leaves each appear in a different $V_k$, $k\in \{1,\ldots,r\}\setminus\{i,j\}$, each of which is called a \emph{small leaf}. The label of the big leaf will be $t$. The label of the center and of each of the small leaves will be $1$. We use $S^{*}$ to denote a member of $S_{t}^{*}(i,j)$ for some $i,j\in \{1,\ldots,r\}$.

Given $S^*\in S_{t}^{*}(i,j)$, we denote by $\chi(S^*)\in \mathbb{R}_{\geq0}^{rn}$ to be the vector with a value of $t$ in the entry corresponding to the vertex of the big leaf, a value of $1$ in the entry corresponding to the vertex of the center, a value of $1$ at the entry corresponding to each vertex that is a small leaf, and a value of $0$ everywhere else.



Therefore, with $\mathcal{S}=\bigcup_{(i,j)} S_t^{*}(i,j)$, a balanced $r$-partite graph $G$ has a perfect fractional $S^*$-tiling if there exists a function $w$ such that
\begin{align}
    \sum_{S^* \in \mathcal{S}} w(S^*)\cdot\lambda_{S^*}(v) = 1, \text{ for all $v\in V(G)$} .\label{ineq:fractional star}
\end{align}

We are ready to state the main result of this section. Lemma~\ref{lemma: fractional tiling} will be used to obtain a perfect fractional $S^*$-tiling of the Szemer\'{e}di graph of $G\in \cG_{r}(\alpha;n)$ for a sufficiently large value of $t$.

\begin{lemma}\label{lemma: fractional tiling}
    Let $r\geq 3$ be a positive a integer, $\alpha >0$, and let $t$ be an integer such that
    \begin{align*}
        t \geq \frac{(r-1) \lfloor (1-\alpha) n\rfloor}{\lceil \alpha n \rceil}.
    \end{align*} Then $G\in \cG_r(\alpha;n)$ admits a perfect fractional $S^*$-tiling.
\end{lemma}

We will use the well known Farkas' Lemma (see \cite[Theorem 8]{martinskoken2017}). For a set $Y \subseteq \mathbb{R}^{N}$ the set $\operatorname{PosCone}(Y)$ denotes the set of all linear combinations of the elements of $Y$ with non-negative coefficients. 
\begin{lemma}[Farkas' Lemma]\label{lem:farkas}
Let $N\geq 1$ be a positive integer, let $Y\subseteq \mathbb{R}^{N}$. Suppose that $v \in \mathbb{R}^{N} - \operatorname{PosCone}(Y)$, then there is some $x \in \mathbb{R}^{N}$ such that 
\begin{itemize}
    \item $x^{T}y \leq 0$ for every $y \in Y$ and,
    \item $x^{T}v >0$. 
\end{itemize}

\end{lemma}

\begin{proof}[Proof of Lemma~\ref{lemma: fractional tiling}]

    The existence of a function $w$ that satisfies \eqref{ineq:fractional star} is equivalent to $\mathbf{1} \not \in \operatorname{PosCone}(Y)$ where $Y:= \{\chi(S^{*}) :S^{*} \in \mathcal{S} \}$.
    
    For a contradiction, suppose that $\mathbf{1} \not \in \operatorname{PosCone}(Y)$. 
    By Farkas' Lemma~\ref{lem:farkas}, there exists $x \in \mathbb{R}^{rn}$ such that 
    \begin{itemize}
        \item $x^{T} \chi(S^{*})\leq 0$ for all $S^{*} \in \mathcal{S}$ and, 
        \item $x^{T} \mathbf{1} >0$.
    \end{itemize}
    We order the vertices within each part according to this $x$ vector: If $V(G) = V_1 \sqcup V_2 \sqcup \cdots \sqcup V_r$, order the vertices of $V_i$ by $V_i = \{v_{i}^{1}, \dots, v_{i}^{n}  \}$ such that $x^{T} \mathbf{1}_{G}(\{v_{i}^{a}\}) \geq x^{T}\mathbf{1}_{G}(\{v_{i}^{b}\})$ whenever $a \leq b$.  Let $V_i = X_i \sqcup Y_i$ where $X_{i}$ contains the first $\lfloor(1-\alpha)n\rfloor$ largest vertices and $Y_{i}$ contains the remaining $\lceil\alpha n\rceil$. 
    
   Let $c = \lfloor (1-\alpha)n\rfloor +1$.
   Given an ordered pair $(i_1, i_2)$ where $i_1 \not = i_2$ and $i_1,i_2 \in [r]$, the $S^* \in S^*_{t}(i_1,i_2)$ that we use will be formed as follows: First we use the vertex $v^1_{i_1}$ as the big leaf. 
   Then we choose its first neighbor in $V_{i_2}$ in the ordering above to be its center, which precedes $v^c_{i_2}$ because of the minimum degree condition. 
   Finally, for each $j\in\{1,\ldots,r\}\setminus\{i_1,i_2\}$, choose the first neighbor of the center. Again, each of these leaves will precede $v^c_{j}$. 
   
   As a result, with respect to $x$, the vector $\chi(S^*)$ will dominate the vector with entry $t$ for $v^1_{i_1}$, entry $1$ for each $v^c_{j}$, $j\neq i_1$ and entry $0$ otherwise. Thus,
   \begin{align*}
       \frac{\lceil\alpha n\rceil}{r-1}x^{T}\chi(S^{*})
       \geq \frac{\lceil\alpha n\rceil}{r-1} \Biggl(  t\cdot x^{T}\mathbf{1}_{G}\bigl(\bigl\{v_{i_1}^{1}\bigr\}\bigr) + x^{T}\mathbf{1}_{G}\bigl(\bigl\{v_{i_2}^{c}\bigr\}\bigr) + \sum_{j \not \in \{i_1, i_2 \}} x^{T}\mathbf{1}_{G}\bigl(\bigl\{v_{j}^{c}\bigr\}\bigr) \Biggr).
   \end{align*}

   Recall that $t\geq (r-1)\lfloor(1-\alpha)n\rfloor/\lceil\alpha n\rceil$.
      \begin{align*}
       \lefteqn{\frac{\lceil\alpha n\rceil}{r-1}x^{T}\chi(S^{*})} \\
       &\geq \lfloor(1-\alpha)n\rfloor x^{T}\mathbf{1}_{G}\bigl(\bigl\{v_{i_1}^{1}\bigr\}\bigr) + \frac{\lceil\alpha n\rceil}{r-1}x^{T}\mathbf{1}_{G}\bigl(\bigl\{v_{i_2}^{c}\bigr\}\bigr) + \frac{\lceil\alpha n\rceil}{r-1}  \sum_{j \not \in \{i_1, i_2 \}} x^{T}\mathbf{1}_G\bigl(\bigl\{v_{j}^{c} \bigr\}\bigr)\\
       &\geq \bigl|X_{i_1}\bigr| x^{T}\mathbf{1}_{G}\bigl(\bigl\{v_{i_1}^{1}\bigr\}\bigr) + \frac{1}{r-1} \bigl|Y_{i_2}\bigr| x^{T}\mathbf{1}_{G}\bigl(\bigl\{v_{i_2}^{c}\bigr\}\bigr) + \frac{1}{r-1}  \sum_{j \not \in \{i_1, i_2 \}} \bigl|Y_{j}\bigr| x^{T}\mathbf{1}_G\bigl(\bigl\{v_{j}^{c}\bigr\}\bigr)\\
       &\geq x^{T}\Biggl(\mathbf{1}_{G}(X_{i_1}) + \frac{1}{r-1}\mathbf{1}_G{(Y_{i_2})} + \frac{1}{r-1} \sum_{j \not \in \{i_1, i_2 \}} \mathbf{1}_G{(Y_j)}\Biggr)
\end{align*}

   We have $r(r-1)$ of these $S^*$'s, one for each pair $(i_1,i_2)$. Summing over the pairs,
    \begin{align*}
        &\sum_{(i_1, i_2)} \Bigl(\mathbf{1}_{G}(X_{i_1}) + \frac{1}{r-1} \mathbf{1}_{G}(Y_{i_2}) + \frac{1}{r-1}\sum_{j \not \in \{i_1,i_2 \}} \mathbf{1}_{G}(Y_{j}) \Bigr)\\
        &= (r-1) \sum_{j=1}^{r} \mathbf{1}_{G}(X_{j}) + (r(r-1) - (r-1))\frac{1}{r-1} \sum_{j=1}^{r}\mathbf{1}_{G}(Y_j)
        = (r-1) \mathbf{1}.
    \end{align*}
    Now, multiplying by $x^{T}$ we obtain
    \begin{align*}
        0 < (r-1)x^{T} \mathbf{1} \leq \frac{\lceil \alpha n\rceil}{r-1} \sum_{(i_1, i_2)} x^{T} \chi(S^{*}) \leq 0,
    \end{align*}
    a contradiction.
\end{proof}

\section{Extremal example}\label{s:extremal example}
We prove that Theorem~\ref{thm: main} satisfies Definition~\ref{def:threshold}~\ref{def:th2} by providing a $G_{n}^{*} \in \mathcal{G}(r;\alpha)$ and a constant $C>0$ such that if $p\leq Cn^{-2/r}$, no perfect $K_r$-tiling exists in $G_n^{*}\cup G_r(n,p)$ \whp. 

Let $\beta =1-\alpha$, and let $G= G_{n}^{*}$ have vertex classes $V_i = A_i \sqcup B_i$ and $|B_i| = \beta n$ for each $1\leq i \leq r$. Next, $G$ is defined to have all edges in each of the pairs $\bigl(A_i, B_j\bigr)$, $\bigl(A_i, A_j\bigr)$ and no edges in the pair $\bigl(B_i, B_j\bigr)$ for all distinct $i,j$. 
By way of contradiction, suppose that $G' = G \cup G_r(n,p)$ contains a perfect $K_r$-tiling. Now let, $\eta:= 1-r\alpha >0$ and $\epsilon :=  (r-1)\alpha/(1-\alpha)$. The number of copies of $K_r$ not using at least one vertex in $A_1 \sqcup \dots \sqcup A_r$ is at most
\begin{align*}
    n - r\alpha n = (1-r\alpha )n=\eta n.
\end{align*}
The proportion of the number of vertices that cannot be covered by the deterministic edges is at least
\begin{align*}
\frac{ \beta n - \eta n}{\beta n} = \frac{(r-1)\alpha}{1-\alpha}= \epsilon.
\end{align*}
Then by applying Lemma~\ref{lemma: partial covering} to $G_{r}(n,p)\bigl[B_1 \sqcup \dots \sqcup B_r\bigr] \cong G_{r}(\beta n,p)$, it guarantees that, \whp, no $K_r$-tiling in $G_r(n,p)\bigl[B_1 \sqcup \dots \sqcup B_r\bigr]$ exists, hence no perfect $K_r$-tiling in $G'$ exists.

\section{Proof of the main theorem}\label{s:proof of main}
In order to prove Theorem~\ref{thm: main}, we will show that there exists a constant $C>0$ such that if $p\geq Cn^{-2/r}$, then the graph $G_{n}\cup G_r(n,p)$ contains a perfect $K_r$-tiling \whp. We provide an outline before proceeding with the proof.
\begin{itemize}
    \item[\ref{sss:regularity}] Apply the  Regularity Lemma (Lemma~\ref{lemma: multipartite regularity}) to obtain a ``cleaned up'' spanning subgraph $G'$ of $G$. Obtain a minimum degree condition for the Szemer\'{e}di graph $\GSz$. 
    \item [\ref{sss:star tiling}] Obtain an $S^{*}$-tiling of $\GSz$.
    \item [\ref{sss:cleaning up}] Remove some extra vertices from each cluster so that the center of each star is super-regular with each of its leaves.
    Ensure that all clusters have sizes divisible by $t$.
    \item[\ref{sss:leftover}] Obtain a partial $K_r$-tiling of $G'$ that contains all of the leftover vertices.
    \item[\ref{sss:balancing}] Obtain another partial $K_r$-tiling of $G'$ (vertex-disjoint from the previous one), so that after removing its vertices, each cluster has the same size which is divisible by $t$.
    \item[\ref{sss:partition}] Partition the existing stars to create a perfect $K_{1,r-1}$-tiling of the remaining clusters of $\GSz$. 
    \item[\ref{sss: tiling K_{1,r-1}}] Find a perfect $K_r$-tiling within the vertices of each copy of $K_{1,r-1}$ in $\GSz$.
\end{itemize}

\subsubsection{Applying the Regularity Lemma}\label{sss:regularity}
Choose
\begin{align*}
0<\eta \ll \epsilon \ll \epsilon_1 \ll \epsilon_2 \ll \epsilon_3 \ll d \ll \alpha \ll 1/r.   
\end{align*}
 We apply the multipartite degree form of Szemer\'edi's Regularity Lemma (Lemma~\ref{lemma: multipartite regularity}) to $G=G_n \in \cG_r(\alpha;n)$, with parameters $\epsilon$ and $d$ to obtain a spanning subgraph $G'$ and a partition of each $V_i$ into $\ell \leq M(\epsilon)$ parts,  $V_i = V_i^{0} \sqcup V_i^{1} \sqcup \dots \sqcup V_{i}^{\ell}$ which satisfy the properties of Lemma~\ref{lemma: multipartite regularity}. In particular, for each $i \in [r]$ and $j \in [\ell]$ we have $(1-\epsilon)n/\ell \leq \bigl|V_{i}^{j}\bigr|= L \leq n/\ell$. Finally, by choosing $\epsilon \ll d$ we have that $\deg_{G'}(v,V_{i}) \geq (\alpha - 2d)n$ for each $v \not \in V_i$, which implies that $\deg(v, V_{i}^{j}) \geq (\alpha - 2d)L$ for each $j \in [\ell]$.
The following holds:
 \begin{enumerate}
     \item $V_i = \bigsqcup_{j=0}^{\ell} V_{i}^j$ for each $i \in [r]$.
     \item $\frac{(1-\epsilon)}{\ell}n\leq \bigl|V_{i}^{j}\bigr| = L \leq \frac{n}{\ell}$ for each $i \in [r]$ and $j \in [\ell]$.
     \item $\deg\bigl(v, V_i^{j}\bigr) \geq (\alpha -2d)L$ for all $v \not\in V_{i}^{j}$, $i \in [r]$, and $j \in [\ell]$.
     \item All pairs $\bigl(V_{i}^{j}, V_{i'}^{j'}\bigr)$ are $\epsilon$-regular with density either at least $d$ or equal to 0. 
 \end{enumerate}

Now, define the Szemer\'edi graph $\GSz$ of $G'$ which has vertex set equal to the clusters of $G'$ (omitting $\bigcup_{i=1}^{r}V_{i}^{0}$ since we don't consider leftover sets to be clusters) and edge $V_{i}^{j} V_{i'}^{j'}$ whenever $\bigl(V_{i}^{j}, V_{i'}^{j'}\bigr)$, $1\leq i \not =i'\leq \ell$ forms an $\epsilon$-regular pair with density at least $d$. Clearly $\GSz$ is a balanced $r$-partite graph on $r \ell$ vertices. Moreover, by Lemma~\ref{lemma: reduced graph degree}, we have that $\delta^{*}(\GSz) \geq (\alpha - d/2 - 3\epsilon)\ell \geq (\alpha/2)\ell$. 

\subsubsection{Tiling the Szemer\'edi graph}\label{sss:star tiling} 
Recall from Section~\ref{ss:lp prelim} that the set $S_t^{*}(i,j)$ will denote the set of all labeled copies of $K_{1,r-1}$ with the center vertex in $V_i$ and the big leaf in $V_j$. We use $S^{*}$ to denote a member of $S_{t}^{*}(i,j)$ for some $i,j$.

For integers $t\geq 2(r-1)/\alpha$, and $i,j\in [r]$, the set $S_{t}(i,j)$ will denote the set of all copies of $K_{1,t+r-2}$ with \emph{center vertex} in $V_i$, $t$ leaves in $V_j$, which we will call \emph{big leaves}, and one leaf in each of the $r-2$  remaining color classes, which we will call \emph{small leaves}. We will use $S$ to denote a member of $S_{t}(i,j)$. Moreover, we will call a collection of vertex-disjoint copies of $S$, where $S \in \bigcup_{(i,j)} S_{t}(i,j)$, an $S$-tiling.

Note that an $S$-tiling refers to a subgraph of $\GSz$ consisting of vertex-disjoint copies of $K_{1,t+r-2}$, where $t$ leaves are in the same vertex class. However, an $S^{*}$-tiling refers to a fractional tiling by copies of $K_{1,r-1}$ in which one leaf is assigned label $t$ and each of the other $r-1$ vertices is assigned label $1$.

Claim~\ref{claim: S^{*}-tiling} below states that we can use a $S^{*}$-tiling of the Szemer\'edi graph $\GSz$ to create a new Szemer\'edi graph $\tGSz$ for $G'$ such that $\tGSz$ admits a perfect $S$-tiling. In doing so, we will increase the number of clusters, and decrease the size of the clusters of $G'$.

\begin{claim}\label{claim: S^{*}-tiling}
    There exists an $\ell_1= \ell_1(\epsilon)$ and a balanced $r$-partite Szemer\'edi graph $\tGSz$ on $r\ell_1$ vertices such that $\tGSz$ admits a perfect $S$-tiling.
\end{claim}
\begin{proof}[Proof of Claim~\ref{claim: S^{*}-tiling}.]
    Recall that $\delta^{*}(\GSz) \geq (\alpha/2) \ell$ and $t$ is a positive integer which satisfies:
    \begin{align*}
    t \ge \dfrac{2(r-1)}{\alpha} > \dfrac{(r-1)(1-\alpha/2)\ell}{\lceil (\alpha/2)\ell \rceil}.
\end{align*}
We apply Lemma~\ref{lemma: fractional tiling} to $\GSz$ in order to obtain a perfect fractional $S^{*}$-tiling $\mathcal{S}$ of $\GSz$.

Let $\mathcal{S}^{+}$ denote the members of $\mathcal{S}$ with positive weights. We will partition the clusters $V_{i}^{j}$ uniformly at random according to $\mathcal{S}^{+}$. 
Consider a weight function $w$ corresponding to the solution that achieves equality in \eqref{ineq:fractional}. We may assume that $w(S^{*})$ is rational for each $S^{*} \in \mathcal{S}$ (see \cite{chvatal1983linear}, Theorem 18.1). Let $D(\GSz)$ be the greatest common denominator of the set of all entries of $w(S^{*})$ for each $S^{*} \in \mathcal{S}^{+}$. Since \eqref{ineq:fractional} depends only on $\GSz$ and the number of Szemer\'edi graphs depends only on $\epsilon$ and $r$, we can find an integer $D=D(\epsilon,r)$ which is the the least common multiple of all the gcd's $D(\GSz)$. Therefore, we have that $D\cdot w(S^{*})$ is a positive integer for each $S^{*} \in \mathcal{S}^{+}$.

We will make use of a variant of the Random Slicing Lemma (Lemma~\ref{lemma: random slicing martin-skoken}) to ensure that from the perfect fractional $S^{*}$-tiling of $\GSz$, there exists a perfect $S$-tiling of $\GSz$.
Each cluster $V_{i}^{j}$ is partitioned uniformly at random into $D$ parts $\tilde{V}_{i}^{j}$, each of size $L_1 :=t\lfloor L/tD \rfloor$ and one part of size at most $L - tD\lfloor L/tD \rfloor < tD$ which will be moved to the leftover set $V_{i}^{0}$. 
By Lemma~\ref{lemma: random slicing martin-skoken}, the probability that a pair $\bigl(\tilde{V}_{i}^{j}, \tilde{V}_{i'}^{j'}\bigr)$ is not $\epsilon_1$-regular,  with $\epsilon_1=(16\epsilon)^{1/5}$, is at most $\exp\{-C_{\ref{lemma: random slicing martin-skoken}}DL\}$, so the probability that there exists any pair that is not $\epsilon_1$-regular is at most 
\begin{align*}
    \binom{r}{2} \ell^{2}\cdot\exp{\{-C_{\ref{lemma: random slicing martin-skoken}}DL\}} \leq \frac{r^2}{2}M^2\cdot\exp{\{-C_{\ref{lemma: random slicing martin-skoken}}DL\}}.
\end{align*}
Hence, each pair $\bigl(\tilde{V}_{i}^{j}, \tilde{V}_{i'}^{j'}\bigr)$ is $\epsilon_1$-regular \whp, since $\ell \leq M$ and $L \geq (1-\epsilon)n/M$. 

For each $i,j$ and $S^{*}\in S_{t}^{*}(i,j)$, construct $w(S^{*})\cdot D$ copies of $S$ by arbitrarily choosing a center of one of the clusters in $V_i$, $t$ vertices of $V_j$, and one leaf in each of $V_k$, $k \in \{1,\ldots, r\}\setminus \{i,j\}$.
Therefore, for $n$ sufficiently large, the desired partition exists. This concludes the proof of Claim~\ref{claim: S^{*}-tiling}.
\end{proof}

Note that the number of clusters (not including the leftover set) in $V_i$ is exactly $\ell_1=\ell\cdot D$. We denote this new graph $\tGSz$ with clusters $\tilde{V}_{i}^{j}$ and leftover set $\bigcup_{i=1}^{r}V_{i}^{0}$, where $i \in \{1,\ldots, r\}$ and $j \in \{ 1,\ldots, \ell_1\}$. 
After adding the discarded vertices to the respective leftover set $V_i^0$ we have that $|V_{i}^{0}| \leq \epsilon n + \ell D < 2\epsilon n$ for each $i \in \{1, \ldots, r\}$. 
Therefore:
\begin{enumerate}
    \item $\ell_1 = \ell \cdot D$.
    \item $|V_{i}^{0}|\leq 2\epsilon n$ for each $i \in \{1, \ldots, r\}$.
    \item $|\tilde{V}_{i}^{j}| = L_1 =  t\lfloor L/tD \rfloor$ for all $i \in \{ 1,\ldots, r\}$ and all $j \in \{1, \ldots \ell_1\}$.
    \item $\deg_{G'}(v, \tilde{V}_{i}^{j}) \geq (\alpha - 3d)L_1$ for all $v \in \tilde{V}_{i'}^{j'}$ for each $i'\not = i \in \{ 1,\ldots, r\}$ and $j' \in \{1, \dots, \ell_1\}$.
   \item All pairs $\bigl(\tilde{V}_{i}^{j}, \tilde{V}_{i'}^{j'}\bigr)$ are $\epsilon_1$-regular with density either at least $d_1= d-\epsilon$ or equal to $0$.
   \item There exists a perfect $S$-tiling of $\tGSz$.
\end{enumerate}

\subsubsection{Cleaning up the stars}\label{sss:cleaning up}
For each copy of $S$ in $\tGSz$ we remove some additional vertices with the property that all pairs of clusters in $S$ are $(\epsilon_1, \delta)$-super-regular, where $\Delta=t+r-2$  and $\delta=2\Delta\epsilon_1/\bigl(2+\Delta\epsilon_1\bigr)$. Observe that $\delta<\Delta\epsilon_1<(d-\epsilon_1)/2$ and that $\delta>\Delta\epsilon_1-(\Delta\epsilon_1)^2/2$. 

Removing the vertices is done via applying Lemma~\ref{lemma: super-regularization} to each $S\in S_{t}(i,j)$ for all $i,j\in [r]$. 
Let such a $\delta$ be denoted $\delta_{\ref{lemma: super-regularization}}$. In discarding these vertices, there are exactly $\delta_{\ref{lemma: super-regularization}} L_1$ from each cluster of $\tGSz$. 
Therefore, for each $i \in [r]$, we have  
\begin{align*}
\bigl|V_{i}^{0}\bigr| &\leq 2\epsilon n + \ell_1 L_1 \delta_{\ref{lemma: super-regularization}} \\
&\leq 2\epsilon n + (\ell D) t \lfloor L/(tD) \rfloor (t+r-2) \epsilon_1 \\
&\leq 2\epsilon n + (t+r-2)\epsilon_1 \ell L \\
&\leq (t+r)\epsilon_1 n.
\end{align*} 
We have obtained a partition of the vertex set $V_i=\Bigl(\bigsqcup_{j=1}^{\ell_1}\tilde{V}_i^j\Bigr)\sqcup V_i^0$ and associated Szemer\'edi graph $\tGSz$ with the following properties:
\begin{enumerate}
    \item $|V_{i}^{0}|\leq (t+r)\epsilon_1 n$ for each $i \in \{1, \ldots, r\}$.
    \item 
    $|\tilde{V}_{i}^{j}| = (1-\delta_{\ref{lemma: super-regularization}})L_1$ for all $i \in \{1,\ldots, r\}$ and all $j \in \{1, \ldots, \ell_1\}$.
    \item$\deg_{G'}\bigl(v, \tilde{V}_{i}^{j}\bigr) \geq (\alpha - 3d-\delta_{\ref{lemma: super-regularization}})L_1 \geq (\alpha - 4d)L_1$ for all $v \in V_{i'}^{j'}$ with $\bigl(V_{i}^{j}, V_{i'}^{j'}\bigr)$ $\epsilon_1$-regular with density at least $d_1$ \whp.
    \item All pairs $\bigl(\tilde{V}_{i}^{j}, \tilde{V}_{i'}^{j'}\bigr)$ are $\epsilon_1$-regular with density either at least $d_1= d-\epsilon$ or equal to 0  where $\epsilon_1 = (16\epsilon)^{1/5}$.
    \item Each pair $\bigl(\tilde{V}_{i}^{j}, \tilde{V}_{i'}^{j'}\bigr)$ that forms an edge in the $S$-tiling of $\tGSz$ is $\bigl(\epsilon_1, \delta_{\ref{lemma: super-regularization}}\bigr)$-super-regular with density either at least $d_1= d-\epsilon$ or equal to $0$, where $\delta_{\ref{lemma: super-regularization}} \geq \Delta\epsilon - \bigl(\Delta\epsilon\bigr)^2/2$, where $\Delta=t+r-2$.
    \item $\tGSz$ admits a perfect $S$-tiling.
\end{enumerate}

\subsubsection{Tiling leftover vertices}\label{sss:leftover}
We find a $K_r$-tiling which covers all  of $\bigcup_{i=1}^{r} V_{i}^{0}$ \whp.  
\begin{lemma}\label{lemma: leftover tiling}
There exists a partial $K_r$-tiling $\mathcal{T}_1$ of $G'$ such that for any cluster $A$, we have that the following holds \whp~for $A'=A\setminus V(\mathcal{T}_1)$:
\begin{itemize}
    \item $\bigcup_{i=1}^{r}V_{i}^{0} \subset V(\mathcal{T}_1)$, 
    \item $(1-2\epsilon_2)|A| \leq |A'| \leq |A|$ for every cluster $A$, and 
    \item if $(A,B)$ is an $(\epsilon_1, \delta_{\ref{lemma: super-regularization}})$-super-regular pair in $G'$, then $(A',B')$ is $(\epsilon_2, \delta_{\ref{lemma: super-regularization}}/2 )$-super-regular for $\epsilon_2 = O(\epsilon_1/\alpha^2)$. 
\end{itemize}
\end{lemma}
The following is a sketch of the proof of Lemma~\ref{lemma: leftover tiling}. We omit the details as the proof follows along the same lines as \cite[Claim 5.2]{balogh2019tilings} and \cite[Section 5]{gomezleos2024tilingrandomlyperturbedbipartite}.

\begin{proof}[Outline of proof of Lemma~\ref{lemma: leftover tiling}]
\renewcommand{\theenumi}{(\it\alph{enumi})}~\\
\begin{enumerate}
     \item For each $i\in [r]$, assign to $v \in V_{i}^{0}$, a set of $(r-1)$ additional clusters, each from a separate color class, $A_j$ for all $j\neq i$, such that $v$ has many neighbors in each $A_j$.
     \item For each $v \in V_{i}^{0}$ and each $j\neq i$, the set $N_j(v)\subset A_j$ will consist of the neighbors of $v$ in $A_j$.
     \item For each $v\in V_{i}^{0}$, find a copy of $K_{r-1}$ in the random edges induced by $N_j(v)$, for $j\neq i$. These will all be pairwise vertex-disjoint for all $v\in\bigcup_{i=1}^{r} V_{i}^{0}$ and all $i\in[r]$. 
     \end{enumerate}
\renewcommand{\theenumi}{(\it\roman{enumi})}
\end{proof}

Upon removing $\mathcal{T}_1$, we have:
\begin{enumerate}
    \item $(1-2\epsilon_2)L_1 \leq \bigl|\tilde{V}_{i}^{j}\bigr| \leq L_1 $ for all $i \in \{ 1,\ldots, r\}$ and all $j \in \{1, \ldots, \ell_1\}$.
    \item Each pair $\bigl(\tilde{V}_{i}^{j}, \tilde{V}_{i'}^{j'}\bigr)$ that forms an edge in the $S$-tiling of $\tGSz$ is $(2\epsilon_2, \delta_{\ref{lemma: super-regularization}}/2)$-super-regular with density either at least $d_1= d-\epsilon$ or equal to 0.
\end{enumerate}

\subsubsection{Balancing the copies of $S$}\label{sss:balancing}
At this point we have left to tile only the vertices belonging to those clusters that were matched by the $S$-tiling. To this end, we find a partial $K_r$-tiling which covers at most $t \lceil 2\epsilon_2 L_1 /t \rceil$ vertices from each cluster. This is done in order to make all clusters the same size $t\lfloor(1-2\epsilon_2)  L_1 /t\rfloor$. This is accomplished by Lemma~\ref{lemma: balance sizes}.
\begin{lemma}\label{lemma: balance sizes} \Whp~there exists a partial $K_r$-tiling $\mathcal{T}_2$ of $G'$ such that for all clusters $A$, the subset $A' = A \setminus V(\mathcal{T}_2)$ satisfies the following:
\begin{itemize}
    \item $|A'| =L_2:= t \lfloor(1-2\epsilon_2)  L_1 /t\rfloor$,
    \item For each pair $(A, B)$ that forms an edge in the $S$-tiling of $\tGSz$, the pair $(A',B')$ is $(4\epsilon_2, \delta_{\ref{lemma: super-regularization}}/4)$-super-regular with density either at least $d_1/2$ or equal to $0$. 
\end{itemize}
\end{lemma}
\begin{proof}[Proof of Lemma~\ref{lemma: balance sizes}]
We will now make all of the leaves the same size $L_2$ by grouping together clusters of size greater than $L_2$ into collections of size $r$ and making use of random edges.

Since all of the color classes have $\ell_1$ clusters, then for every $\tilde{V}_{i}^{j}$ with size exceeding $L_2$, we can find $\tilde{V}_{i_1}^{j_1}, \ldots, \tilde{V}_{i_{r-1}}^{j_{r-1}}$ also with size exceeding $L_2$.
While $|\tilde{V}_{i}^{j}| -L_2 >0$ we use Lemma~$\ref{lemma: partial covering}$ to remove copies of $K_{r}$ greedily. By the Slicing Lemma (Corollary~\ref{lem: super-regular slicing}), if we previously had that $(A, B)$ was $(2\epsilon_2, \delta_{\ref{lemma: super-regularization}} /2)$-super-regular, then $(A',B')$ is $(4\epsilon_2, \delta_{\ref{lemma: super-regularization}} /4)$-super-regular. 
Note that the clusters  are all of size $L_2$, which is divisible by $t$.
\end{proof}

\subsubsection{Partitioning the $S$ stars}\label{sss:partition}
Now, we have tiled all of the leftover vertices. Each cluster of $\tGSz$ is of size $L_2$. Moreover, there exists a perfect $S$-tiling of $\tGSz$. 

Among the big leaves, we will find a $K_{r}$-tiling that covers approximately a $1-1/t$ proportion of each big leaf. Upon finding this tiling and making some small alternations, we then obtain clusters, each of the same size (approximately $L_2/t$) that are grouped into disjoint sets of $r$ clusters with a $K_{1,r-1}$ structure in which one ``center'' vertex is super-regular with the other $r-1$ of them. 



For each copy $S$ of $S_{t}(i,j)$ in $\tGSz$ and each $i$ in which $S$ has a center or a small leaf, let $U_{i} = U_{i}(S)$ be that cluster. 
Moreover, for the $i$ for which $S$ has the big leaves, let the big leaf be $T_k = T_k(S)$ for $k \in \{1,\ldots, t\}$:
\begin{figure}
\centering
\begin{tikzpicture}[scale=.5]
\def\va{-10.0};
\def\vb{-6.0};
\def\vc{-2.0};
\def\vd{2.0};
\def\ve{6.0};
\def\vf{10.0};
\def\elw{1.7cm};
\def\elh{6.5cm};
\def\recw{1.5};
\def\rech{6.0};
\useasboundingbox (-12,-7.0) rectangle (12,8);

\def\ya{4.0};

\begin{pgfonlayer}{back}
\draw[thin,rounded corners] ($(\va,0)+(-\recw,-\rech)$) rectangle ($(\va,0)+(\recw,\rech)$);
\draw[thin,rounded corners] ($(\vb,0)+(-\recw,-\rech)$) rectangle ($(\vb,0)+(\recw,\rech)$);
\draw[thin,rounded corners] ($(\vc,0)+(-\recw,-\rech)$) rectangle ($(\vc,0)+(\recw,\rech)$);
\draw[thin,rounded corners] ($(\vd,0)+(-\recw,-\rech)$) rectangle ($(\vd,0)+(\recw,\rech)$);
\draw[thin,rounded corners] ($(\vf,0)+(-\recw,-\rech)$) rectangle ($(\vf,0)+(\recw,\rech)$);
\node at (6,0) {\textbf{$\ldots$}};
\node at (\va,-6.6) {$V_1$};
\node at (\vb,-6.6) {$V_2$};
\node at (\vc,-6.6) {$V_3$};
\node at (\vd,-6.6) {$V_4$};
\node at (\vf,-6.6) {$V_r$};
\end{pgfonlayer}

\newcommand{\textmaker}[1]{\parbox{0.8cm}{\centering #1}}
\begin{pgfonlayer}{fore}
\node[circle,draw,ultra thick,fill=white] (U1) at (\va,\ya) {\textmaker{$U_1$}};
\node[circle,draw,ultra thick,fill=white] (T1) at (\vb,\ya) {\textmaker{$T_1$}};
\node[circle,draw,ultra thick,fill=white] (Tk) at (\vb,0) {\textmaker{$T_k$}};
\node[circle,draw,ultra thick,fill=white] (Tt) at (\vb,-\ya) {\textmaker{$T_t$}};
\node[circle,draw,ultra thick,fill=white,text width=1.1cm,align=center,inner sep=0pt] at (\vf,\ya) {\footnotesize small\\[-.5em]leaf};
\node[circle,draw,ultra thick,fill=white,text width=1.1cm,align=center,inner sep=0pt] at (\vc,\ya) {\footnotesize small\\[-.5em]leaf};
\node[circle,draw,ultra thick,fill=white,text width=1.1cm,align=center,inner sep=0pt] at (\vd,\ya) {\footnotesize small\\[-.5em]leaf};
\end{pgfonlayer}

\begin{pgfonlayer}{main}
\draw[black, dashed,ultra thick] (U1) -- (T1);
\draw[black, dashed,ultra thick] (U1) -- (Tk);
\draw[black, dashed,ultra thick] (U1) -- (Tt);
\draw[black, dashed,ultra thick] (U1) to[out=40,in=140] (\vc,\ya);
\draw[black, dashed,ultra thick] (U1) to[out=40,in=145] (\vd,\ya);
\draw[black, dashed,ultra thick] (U1) to[out=40,in=155] (\vf,\ya);
\end{pgfonlayer}

\end{tikzpicture}
\caption{An instance of an $S\in S_{t}(1,2)$ in $\tGSz$. The center cluster is $U_1=U_1(S)$ and is in $V_1$. The big leaves are $T_1, \ldots, T_t$ and are in $V_2$. In each of $V_3,\ldots,V_r$, there is a small leaf.}
\label{fig:partition of S}
\end{figure}
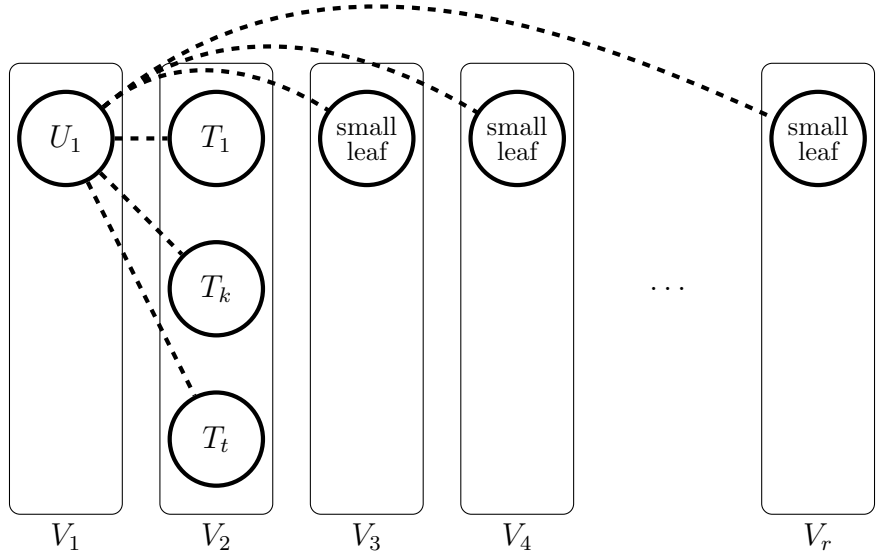

\begin{figure}
\centering
\begin{tikzpicture}[scale=.5]
\useasboundingbox (-12,-3) rectangle (3,5);
\draw[ultra thick] (0,0) circle (3cm);
\draw[ultra thick] (-9,0) circle (3cm);

\node at (0, 2.1) {$T_{k}'$};
\node at (4,-2) {{\Large$T_k$}};

\node at (0,0) {$T_{k}''$};
\node at (3.9,1.5) {$T_{k}'''$};
\node at (-9,1.9) {$U_{1,k}$};
\node at (-13,-2) {{\Large$U_{1}$}};

\draw[->] (3.2,1.5)   -- (2.2,1.2);

\draw[black]  (-2.6,1.5) -- (2.6,1.5);
\draw[black,dashed]  (-2.7,1.0) -- (2.7,1.0);

\draw[black]  (-11.85,1.0) -- (-6.15, 1.0);

\draw[dotted, black,]  (-9,3) -- (0,3);
\draw[dotted, black]  (-6.3,1.0) -- (-2.7,1.0);
\node at (0, 4) {\textbf{big leaf}};
\node at  (-9,4) {\textbf{center}};
\end{tikzpicture}
\caption{The partitioning of a pair $(U_1, T_k)$ into $t$ many pairs  $(U_{1}, T_k)$ for $1\leq k\leq t$. The region $T_{k}'$ bounded by the solid lines is $1/s$ proportion of $T_k$. The region $T_{k}''$ is of size $1-1/s$ proportion of $T_k$. The random tiling $\mathcal{T}_3$ will leave uncovered the vertices $T_{k}'''$. The pair $\bigl(U_{1,k}, T_{k}' \cup T_{k}'''\bigr)$ is shown to be $(5\epsilon_3, \delta_{\ref{lemma: super-regularization}} /5)$-super-regular.}
\label{fig:center big}
\end{figure}
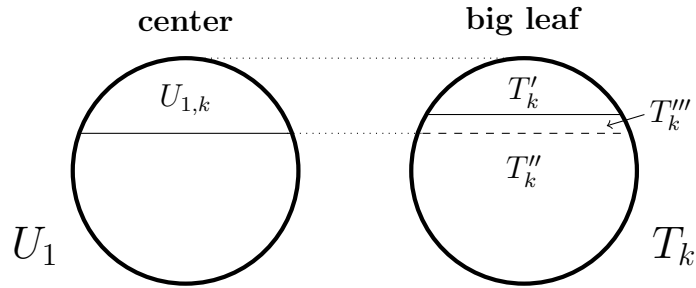

\begin{itemize}
    \item Partition $U_i$ uniformly at random into $t$ parts $U_{i,k}=U_{i,k}(S)$ for $k \in \{1,\ldots, t\}$ such that $|U_{i,k}| = L_3:= L_2/t$. See Figure~\ref{fig:partition of S}.
    \item Let $\eta \ll \epsilon$ and $s$ satisfy $\eta = (\frac{1}{t} - \frac{1}{s})\frac{s}{s-1}$.
    Note $s>t$. We partition each of the big leaf clusters $T_k$ uniformly at random into 2 parts, $T_k'=T_k'(S)$ and $T_k''=T_k''(S)$ with $|T'_{k}| = L_2-\lceil\tfrac{s-1}{s}L_2\rceil$ and $|T''_{k}| =\lceil\tfrac{s-1}{s}L_2\rceil$.  Note that $|T_{k}'|$ is slightly smaller than $L_3=L_2/t$, but we will set aside additional vertices from $T_{k}''$ so that when added to $T_{k}'$, the number of vertices is exactly $L_3$. See Figure~\ref{fig:center big}.   
  \end{itemize}
\begin{lemma}\label{lemma: partial tiling large clusters}
    \Whp~there exists a partial $K_r$-tiling, $\mathcal{T}_3$ of the vertices of the union of all $T_{k}''$ where $T_{k}$ is a big leaf and $T_{k}'' \subseteq T_{k}$. 
    \begin{itemize}
        \item 
        Let $T_{k}'''$ be the vertices not covered by $\mathcal{T}_3$. Then $|T_{k}'''| = |T_k''|-\tfrac{t-1}{t}L_2\geq \eta |T_k''|$.
        \item For each copy $S$ of $S_t(i,j)$, where  $T_k= T_k(S)$ is a big leaf and $U_{i,k}=U_{i,k}(S)$ is the center, if  $W_{k} := T_{k}' \sqcup T_{k}'''$ then $|W_{k}|=|U_{i,k}|=L_3$ and 
        $(W_{k}, U_{i,k})$ is $(5\epsilon_3,\delta_{\ref{lemma: super-regularization}}/ 5)$-super-regular, where $\epsilon_3 :=(64\epsilon_2)^{1/5}$. 
    \end{itemize}
\end{lemma}
\begin{proof}[Proof of Lemma~\ref{lemma: partial tiling large clusters}]
Since there are an equal number of big leaves in each $V_i$, we can arbitrarily group together $r$ many big leaves, each in a different $V_i$. 
To that end, choose big leaves $T_1, \ldots, T_r$ such that $T_i \in V_i$. Uniformly at random, select $T_{i}''$ be a set of size $\bigl(1-1/s\bigr)\bigl|T_{i}\bigr|$, where $\bigl|T_{i}\bigr| = L_2$.

Apply Lemma~\ref{lemma: partial covering} with $n=\bigl\lceil \frac{s-1}{s}L_2\bigr\rceil$ and $\epsilon = \eta$. If $p \geq C_{\ref{lemma: partial covering}}(\eta,r)\times \bigl(\bigl\lceil\tfrac{s-1}{s}L_2\bigr\rceil\bigr)^{-2/r}$, then we find a partial tiling, $\mathcal{T}_3$ of $T_{1}'', \ldots, T_{r}''$, which covers exactly $\frac{t-1}{t}L_2$ 
of each $T_i''$ for each $i \in \{1, \ldots, r\}$. Fortunately,  $p\geq Cn^{-2/r}$ in Theorem~\ref{thm: main} suffices to find such a tiling, provided that $C$ is sufficiently large.

Set $W_{k} = T'_{k} \sqcup T'''_k$, then  
\begin{align*}
     \bigl|W_{k}\bigr| = \Bigl(L_2-\Bigl\lceil\frac{s-1}{s}L_2\Bigr\rceil\Bigr) + \Bigl(\Bigl\lceil\frac{s-1}{s}L_2\Bigr\rceil-\frac{t-1}{t}L_2\Bigr)= \frac{L_2}{t} = L_3.
 \end{align*}
Recall, we previously had that every adjacent pair of clusters in $\tGSz$ is $\bigl(4\epsilon_2, \delta_{\ref{lemma: super-regularization}}/4\bigr)$-super-regular. 
We want to show that $\bigl(U_{i,k}, W_{k}\bigr)$ is $\bigl(5\epsilon_3, \delta_{\ref{lemma: super-regularization}} /5\bigr)$-super-regular.

In order to verify that the subset density condition~\ref{def: sr density} of Definition~\ref{def: super regular} holds, let $X \subseteq U_{i,k}$ and $Y \subseteq W_{k}$ such that
\begin{align*}
    |X| &\geq 5\epsilon_3 L_3 \\
    |Y| &\geq 5\epsilon_3 L_3.
\end{align*}
Let $Y = Y' \sqcup Y'''$ such that $Y' = Y \cap T'_{k}$ and $Y'''=  Y \cap T_{k}'''$. 
Recall that for the pairs $\bigl(U_{i,k}, T_k'\bigr)$ that were obtained by Lemma~\ref{lemma: random slicing martin-skoken}, we have that \whp, $\bigl(U_{i,k},T_{k}'\bigr)$ is $\epsilon_3 =\bigl(64\epsilon_2\bigr)^{1/5}$-regular with density at least $d_1/4 $. 
Note that $|U_{i,k}| > |T_{k}'|$, so Lemma~\ref{lemma: random slicing martin-skoken} as stated does not apply but since $|U_{i,k}| \approx |T_{k}'|$, this technical detail is left to the reader.
We show that $|Y'''| < \epsilon_{3}L_2 $. 
Otherwise, 
\begin{align*}
\eta\Bigl(1-\frac{1}{s}\Bigr)L_3 =\bigl|T_{k}'''\bigr|\geq \bigl|Y'''\bigr| &\geq  \epsilon_3 L_3. \\
\end{align*}
This is a contradiction for $\eta\ll \epsilon_1 \ll \epsilon_2 \ll \epsilon_3 \ll 1/t$.
Since $|Y'''| <  \epsilon_3 L_3$, then $|Y'| \geq  |Y|-\epsilon_3 L_3$. Therefore, $e(X,Y') > (\delta_{\ref{lemma: super-regularization}}/4) |X||Y'|$, so we obtain: 
\begin{align*}
    d(X,Y) &> \frac{\delta_{\ref{lemma: super-regularization}}}{4} \frac{|Y'|}{|Y|} \\
    &\geq \frac{\delta_{\ref{lemma: super-regularization}}}{4} \frac{|Y|-\epsilon_3 L_3 }{|Y|}\\
   &\geq \delta_{\ref{lemma: super-regularization}} / 5.
\end{align*}

Lastly, the minimum degree condition~\ref{def: sr degree} of Definition~\ref{def: super regular} follows from applying Lemma~\ref{lemma: random slicing min deg}. This concludes the proof of Lemma~\ref{lemma: partial tiling large clusters}.
\end{proof}

For each copy $S$ of $S_t(i,j)$, we assign for each center $U_{i,k}$ and each of the small leaves of $S$ to a big part $W_{k}$ arbitrarily. We therefore obtain a perfect $K_{1,r-1}$-tiling of $\tGSz$. This perfect $K_{1,r-1}$-tiling has the following properties for each copy of $K_{1,r-1}$.
\begin{enumerate}
    \item All clusters are of the same size $L_3$.
    \item The center cluster forms a $(5\epsilon_3, \delta_{\ref{lemma: super-regularization}}/5)$-super-regular pair with each leaf.
\end{enumerate}
At this point, it suffices to tile each such copy of $K_{1,r-1}$ in the tiling independently. In the case that $r=3$, this is handled completely by B\"{o}ttcher, Parczyk, Sgueglia, and Skokan \cite[Lemma 4.1]{bottcher2023triangles}. In the following subsection, we consider all $r\geq 3$.

\subsubsection{Tiling each copy of $K_{1,r-1}$}\label{sss: tiling K_{1,r-1}}
Lemma~\ref{lem: eps-reg star-tiling} is a generalization of~\cite[Lemma 4.1]{bottcher2023triangles} and its proof follows along the same lines. Recall the definition of $G_r(V_1, \ldots,V_r, p)$ in \ref{s:intro}. 
\begin{lemma}\label{lem: eps-reg star-tiling}
    Let $r\geq 3$ be fixed. 
    For any $0<d<1$ there exists $\epsilon=\epsilon_{\ref{lem: eps-reg star-tiling}}(d) >0$ and $C=C_{\ref{lem: eps-reg star-tiling}}(d)>0$ such that the following holds. Let $G= (V_1\sqcup \dots \sqcup V_r;E)$ with $|V_i| =n$ for all $i \in [r]$. Let $(V_1, V_i)$ be $(\epsilon,d)$-super-regular for all $i \in \{2,\ldots, r\}$ and let $p \geq Cn^{-2/(r-1)} (\log{n})^{1/\binom{r-1}{2}}$. Then there exists a perfect $K_r$-tiling in $G \cup G_{r-1}\bigl(V_2,\ldots, V_r,p\bigr)$ \whp. 

\end{lemma}

With Lemma~\ref{lem: eps-reg star-tiling}, the proof of Theorem~\ref{thm: main} is complete.\footnote{Note that Theorem~\ref{thm: main} requires  $p\geq Cn^{-2/r}$. This is larger than the lower bound of $\Omega\bigl(n^{-2/(r-1)} (\log n)^{1/\binom{r-1}{2}}\bigr)$, required by Lemma~\ref{lem: eps-reg star-tiling}.}

That is, we can apply Lemma~\ref{lem: eps-reg star-tiling} to each copy of $K_{1,r-1}$ with parameters $n=L_3$ and $d=\delta_{\ref{lemma: super-regularization}}/5$, and choosing $\epsilon_3$ such that $5\epsilon_3\leq \epsilon_{\ref{lem: eps-reg star-tiling}}$. From this point forward in the proof, we will, for convenience, use $n$ in place of $L_3$. 

Before proving Lemma~\ref{lem: eps-reg star-tiling}, we define two auxiliary graphs and provide an outline. 
\begin{definition}\label{def:auxgraph}
Consider the following auxiliary (hyper)graphs.
\begin{itemize}
    \item Let $F=F_{V_1}[V_2,\ldots,V_r]$ be an auxiliary $(r-1)$-uniform $(r-1)$-partite hypergraph with vertex classes $V_2 \sqcup \dots \sqcup V_{r}$, where an $(r-1)$-tuple $(v_2, \dots, v_{r}) \in V_2 \times \dots \times V_{r}$ is an edge of $F$ if $v_2, \dots, v_{r}$ have at least $(d/2)^{r-1} |V_1|$ common neighbors in $V_1$. For $X\subseteq V_1$, we define $F_{X}=F_{X}[V_2,\ldots,V_r]$ to be the subgraph of $F$ such that each edge $(v_2, \ldots, v_r)\in F$ has the additional property that $v_2,\ldots,v_r$ has at least $(d/2)^{r-1}|X|$ common neighbors in $X$.
    \item Given a (not necessarily perfect) $K_{r-1}$-tiling $M$ of $V_2\sqcup \cdots \sqcup V_r$, let $B=B(V_1,M)$ be an auxiliary bipartite graph with vertex classes $V_1$ and $M$ and the pair $\{v,m\}$, $v \in V_1$ and $m\in M$ is an edge of $B$ if each of the vertices of $m$ is adjacent to $v$.
 \end{itemize}
\end{definition}

Lemma~\ref{lem: auxiliary F graph} below (which is proven in Section~\ref{s:proofs of aux lemmas}) gives some properties of of the auxiliary hypergraph $F$.
\begin{lemma}\label{lem: auxiliary F graph}
    Let $r\geq 3$. For any $\epsilon, d>0$ with $\epsilon\leq (d/2)^{r-2}$, the following holds. Let $G$ be a balanced $r$-partite graph on $V_1 \sqcup \cdots \sqcup V_{r}$, with $|V_i| = n$ for all $i \in \{2,\ldots,r\}$ such that $(V_1,V_i)$ are $(\epsilon,d)$-super-regular with respect to $G$ for all $i \in \{2,\ldots,r\}$. Let $F$ be the auxiliary hypergraph described in Definition~\ref{def:auxgraph}. Then $F= F_{V_1}[V_2, \dots, V_r]$ satisfies the following properties.
\renewcommand{\theenumi}{(\it\roman{enumi})}
    \begin{enumerate}
        \item The minimum degree of $F$ is at least $\bigl(1-(r-2)\epsilon\bigr)n^{r-2}$.\label{high min degree}
        \item If $X \subseteq V_1$ and $|X| \geq (2/d)^{r-2} \epsilon n$, then, for each $i \in \{2,\ldots,r\}$, all but at most $\epsilon n$ vertices from $V_i$ have degree at least $\bigl(1-2(r-2)\epsilon\bigr)n^{r-2}$ in $F_{X} = F_{X}[V_2, \ldots, V_r]$.\label{high min deg F_X} 
    \end{enumerate}
\end{lemma}

Lemma~\ref{lem: eps-reg matching} below (which is also proven in Section~\ref{s:proofs of aux lemmas}) will allow us to obtain a $K_{r-1}$-tiling, $M$, which is a matching in the hypergraph $F=F_{V_1}[V_2, \ldots, V_r]$ for which $(V_1, M)$ forms a super-regular pair with respect to $B(V_1, M)$.
\begin{lemma}\label{lem: eps-reg matching}
    For any $0<d,\delta,\epsilon'<1$ with $2\delta \leq d$ there exists $\epsilon = \epsilon_{\ref{lem: eps-reg matching}}(d,\delta,\epsilon'), C= C_{\ref{lem: eps-reg matching}}(d,\delta,\epsilon')>0$ such that the following holds. Let $G$ be an $r$-partite graph on $V_1 \sqcup  \cdots \sqcup V_{r}$ with $|V_1| = \cdots =|V_r| = n$ and $(V_1, V_i)$ are $(\epsilon,d)$-super-regular with respect to $G$ for each $i \in \{2,\ldots, r\}$. Let $G_{r-1}(V_2,\ldots,V_r,p)$ be a random $(r-1)$-partite graph with $p\geq C n^{-2/(r-1)}$. Then \whp~there exists a $K_{r-1}$-tiling $M$ of size $|M|=(1-\delta)n$ such that the pair $(V_1,M)$ is $\bigl(\epsilon',(d/2)^{r-1}/4\bigr)$-super-regular with respect to the auxiliary bipartite graph $B(V_1,M)$. 
\end{lemma}

The outline of the proof of Lemma~\ref{lem: eps-reg star-tiling} is as follows:
\renewcommand{\theenumi}{(\it\alph{enumi})}
\begin{enumerate}
    \item First find a random matching $M$ in the auxiliary hypergraph $F$ of size $(1-\delta)n$ such that the pair $(V_1,M)$ forms a super-regular pair with respect to the auxiliary bipartite graph $B(V_1,M)$.\label{out:M}
    \item For each $i \in \{2,\ldots, r\}$, let $V_{i}' = V_{i}-V(M)$. Next, find an additional perfect matching $M'$ in the hypergraph $F':= F_{V_1}[V_{2}',\dots, V_{r}']$.\label{out:M'}
    \item Extend $M'$ to a $K_{r}$-tiling of size $\delta n$ by greedily selecting vertices from $V_1$. Denote these vertices to be $V_{1}'$.\label{out:greedy}
    \item By the Slicing Lemma, $(V_1 - V_{1}',M)$ is super-regular (with relaxed parameters) with respect to the auxiliary bipartite graph $B(V_1,M)$. Finish by applying Lemma~\ref{lemma: regular matching}.\label{out:slicing}
\end{enumerate}

\begin{proof}[Proof of Lemma~\ref{lem: eps-reg star-tiling}]

First we determine $\epsilon_{\ref{lem: eps-reg star-tiling}}$ and $C_{\ref{lem: eps-reg star-tiling}}$. To that end, given $0<d<1$, let $\epsilon_{\ref{lemma: regular matching}}=\epsilon_{\ref{lemma: regular matching}}\bigl((d/2)^{r-1}/8\bigr)$ be given by Lemma~\ref{lemma: regular matching}. 
Let $\epsilon_{\ref{lem: eps-reg matching}}=\epsilon_{\ref{lem: eps-reg matching}}\bigl(d,\delta=(d/2)^{r-1},\epsilon_{\ref{lemma: regular matching}}/2\bigr)$ be given\footnote{We will not need to use $C_{\ref{lem: eps-reg matching}}$ because $n$ is sufficiently large and $p=\omega\bigl(n^{-2/(r-1)}\bigr)$, which is asymptotically larger than required by Lemma~\ref{lem: eps-reg matching}.} by Lemma~\ref{lem: eps-reg matching}. 
Let $\epsilon_{\ref{lem: eps-reg star-tiling}}= \min \bigl\{\epsilon_{\ref{lem: eps-reg matching}},\delta^{r-2}/(2r)\bigr\}$ and
let $C_{\ref{lem: eps-reg star-tiling}}=  C_{\ref{lem: robust Hajnal-Sze}}\bigl(1/(2r)\bigr)\times 2\delta^{-2/(r-1)}$.

Let $G$ be a balanced $r$-partite graph on $V_1 \sqcup \dots \sqcup V_r$ with $|V_1| = \dots = |V_r| = n$, such that $(V_1, V_i)$ are $(\epsilon,d)$-super-regular pairs with respect to $G$ for all $i \in \{2,\ldots, r\}$.
We reveal random edges in $G_{r-1}(V_2,\ldots,V_r, p)$ in two rounds as $G_1 \sim G_{r-1}(V_2,\ldots,V_r, p/2)$ and $G_2 \sim G_{r-1}(V_2,\ldots,V_r, p/2)$.

For part~\ref{out:M} in the outline, we apply Lemma~\ref{lem: eps-reg matching} to $G_1$, with edge probability $p/2$, to obtain a partial matching $M$ in $F= F_{V_1}[V_2,\ldots,V_r]$ of size $|M| = (1-\delta)n$ such that the pair $(V_1,M)$ is $\bigl(\epsilon_{\ref{lemma: regular matching}}/2, (d/2)^{r-1}/4\bigr)$-super-regular with respect to $B(V_1,M)$.
Now, for each $i\in\{2,\ldots,r\}$, let $V_{i}' = V_i - V(M)$; that is, the vertices of $V_i$ that are not incident to any hyperedge of $M$. Note that $|V_{i}'| = \delta n$ for all $i\in\{2,\ldots,r\}$.

For part~\ref{out:M'}, let $F':= F_{V_1}[V_{2}', \dots, V_{r}']$ be the subhypergraph induced by $V_{2}', \dots, V_{r}'$. By Lemma~\ref{lem: auxiliary F graph}~\ref{high min degree}, we have that $\delta(F) \geq  (1-(r-2)\epsilon)n^{r-2}$. Therefore, $\delta(F') \geq (\delta n)^{r-2} - (r-2)\epsilon n^{r-2} \geq \bigl(1-\frac{(r-2)\epsilon}{\delta^{r-2}} \bigr)\bigl(\delta n \bigr)^{r-2}$. 
Thus, by Lemma~\ref{lem: robust Hajnal-Sze} below, $p/2\geq C_{\ref{lem: robust Hajnal-Sze}}\bigl(1/(2r)\bigr) \times n^{-2/(r-1)}(\log n)^{1/\binom{r-1}{2}}$ suffices to ensure that \whp~there exists a perfect matching $M'$ in $G_2 \cap F'$.

\begin{lemma}[Han, Hu, and Yang~\cite{han2023robustversionmultipartitehajnalszemeredi}, Theorem 1.4]\label{lem: robust Hajnal-Sze}
    Fix $r\geq 3$. For any $\epsilon>0$, there exists a constant $C=C_{\ref{lem: robust Hajnal-Sze}}(\epsilon)$ such that for any $n \in r\mathbb{N}$ and $p \geq Cn^{-2/(r-1)} (\log{n})^{1/\binom{r-1}{2}}$ the following holds. If $\mathcal{H}$ is an $(r-1)$-partite and $(r-1)$-uniform hypergraph with $\delta(\mathcal{H}) \geq (1-1/r+\epsilon)n^{r-2}$, then \whp~$\mathcal{H}$(p) contains a perfect matching. 
\end{lemma}

Now, for part~\ref{out:greedy}, we assign for each $m\in M'$ a distinct $v \in V_1$ greedily. Since each $m\in M'$ has at least $(d/2)^{r-1}n\geq \delta n$ common neighbors in $V_1$, such an assignment is possible. Denote the set of vertices of $V_1$ covered in this way by $V_{1}'$.

For part~\ref{out:slicing}, note that $|V_{1}'| = |M'| = \delta n$. It follows that by our choice of $\epsilon_{\ref{lemma: regular matching}}$ that the pair $(M, V_1 -V_{1}')$ is $\bigl(\epsilon_{\ref{lemma: regular matching}}, (d/2)^{r-1}/4\bigr)$-super-regular with respect to $B(M,V_1 - V_{1}')$. Since $|V_{1} - V_{1}'| = |M|$, by Lemma~\ref{lemma: regular matching}, there is a perfect matching in $B(M,V_1 - V_{1}')$, and this completes the proof of Lemma~\ref{lem: eps-reg star-tiling}.
\end{proof}







\section{Proofs of the Auxiliary Lemmas}\label{s:proofs of aux lemmas}


\subsection{Proof of Lemma~\ref{lem: auxiliary F graph}}
   
   For part~\ref{high min degree}, without loss of generality let $v_2\in V_2$. By super-regularity, every $v_2\in V_2$ has at least $dn>(d/2)n$ neighbors in $V_1$.
    For $\ell\in\{3,\ldots,r\}$, having chosen $\{v_2,\ldots, v_{\ell-1}\}$, with $\bigl|\bigcap_{i=2}^{\ell-1} N_{V_1}(v_i)\bigr|\geq (d/2)^{\ell-2}n\geq\epsilon n$ by Lemma~\ref{lem: reg pairs not too many low degree]}, all but at most $\epsilon n$ vertices of $V_{\ell}$ have at least $(d/2)\bigl|\bigcap_{i=2}^{\ell-1} N_{V_1}(v_i)\bigr|$ neighbors in $V_1$ in common. This holds for all $\ell \in \{3,\ldots, r\}$ because $(d/2)^{r-2} \geq \epsilon$.  
   Therefore, at least $n^{r-2}-(r-2)\epsilon n\cdot n^{r-3}$ tuples $(v_2,v_3,\ldots,v_r)\in \{v_2\}\times V_3 \times\cdots\times V_r$ have at least $(d/2)^{r-2}n$ common neighbors in $V_1$. Hence the degree of $v_2$ is at least $(1-(r-2)\epsilon)n^{r-2}$. 

   For part~\ref{high min deg F_X}, without loss of generality $i=2$. By Lemma~\ref{lem: reg pairs not too many low degree]}, there are at least $n-\epsilon n$ vertices $v_2\in V_2$ such that $\bigl|N_X(v_2)\bigr|\geq (d/2)|X|$ because $(d/2)|X|\geq \epsilon n$. As above, 
    for $\ell\in\{3,\ldots,r\}$, having chosen $\{v_2,\ldots, v_{\ell-1}\}$, with $\bigl|\bigcap_{i=2}^{\ell-1} N_{X}(v_i)\bigr|\geq (d/2)^{\ell-2}|X|\geq\epsilon n$ by Lemma~\ref{lem: reg pairs not too many low degree]}, all but at most $\epsilon n$ vertices of $V_{\ell}$ have at least $(d/2)\bigl|\bigcap_{i=2}^{\ell-1} N_{X}(v_i)\bigr|$ neighbors in $X$ in common. This holds for all $\ell \in \{3,\ldots, r\}$ because $(d/2)^{r-2} \geq \epsilon$. 
    Therefore, at most $(r-2)\epsilon n\cdot n^{r-3}$ tuples $(v_2,v_3,\ldots,v_r)\in \{v_2\}\times V_3 \times\cdots\times V_r$ have at least $(d/2)^{r-2}|X|$ common neighbors in $X$.
    Furthermore, by part~\ref{high min degree}, there are at most $(r-2)\epsilon n^{r-2}$ tuples not in $F$ that contain $v_2$.
    Thus, the number of tuples $(v_2,v_3,\ldots,v_r)\in \bigl(\{v_2\}\times V_3 \times\cdots\times V_r\bigr)\cap F_X$ is at least  $(1-2(r-2)\epsilon)n^{r-2}$. 
   \qed

\subsection{Proof of Lemma~\ref{lem: eps-reg matching}}

We closely follow the approach of B\"{o}ttcher, Parczyk, Sgueglia, and Skokan \cite[Lemma 8.1]{bottcher2023triangles}.
Given $0<d,\delta, \epsilon'<1$ and $2\delta \leq d$, suppose 
$$\epsilon \ll \beta \ll \gamma \ll d, \delta, \epsilon'. $$
Choose $C>0$ for the application of the tail bounds below and let $p \geq Cn^{-2/(r-1)}$. 
In order to find the desired $K_{r-1}$-tiling $M$, we first begin with $F=F_{V_1}[V_2, \dots, V_r]$ as in Definition~\ref{def:auxgraph} and construct $\tilde{F}\subseteq F$ by considering the underlying $(r-1)$-partite random graph $G_{r-1}\bigl(V_2,\ldots, V_r,p\bigr)$ and including each hyperedge $(v_2,\ldots, v_r) \in V_2 \times \ldots \times V_r$ whenever all ${r-1 \choose 2}$ edges appear. That is, the probability of a hyperedge being in $\tilde{F}$ is $p^{r-1\choose 2}$.

 
\begin{claim}\label{claim:many good}
     With high probability, every $r$-tuple $(X,W_2,\ldots,W_r)$ such that $X\subseteq V_1$, $|X|\geq\beta n$, $W_i\subseteq V_i$ and $|W_i|\geq\delta n$ for all $i\in\{2,\ldots,n\}$ satisfies each of the following:
     \renewcommand{\theenumi}{(\it\roman{enumi})}
    \begin{enumerate}
        \item\label{claim:manygood1} The number of hyperedges in $\tilde{F}\cap F_{V_1}[W_2, \dots, W_r]$ is at least 
        $$ (1 -\epsilon)\bigl\|F_{V_1}[W_2, \dots, W_r]\bigr\|p^{\binom{r-1}{2}} \geq \Bigl(1-\frac{r\epsilon}{\delta^{r-1}}\Bigr)\prod_{i=2}^r|W_i|\,p^{\binom{r-1}{2}} 
        . $$\label{claim: many copies}
        \item\label{claim:manygood2} The number of hyperedges in $\tilde{F}\cap F_{X}[W_2, \dots, W_r]$ is at least $\bigl(1-\frac{2r\epsilon}{\delta^{r-1}}\bigr)\prod_{i=2}^r|W_i|\, p^{\binom{r-1}{2}}$ (that is, the $(r-1)$-tuples belonging to these edges have at least $(d/2)^{r-1}|X|$ common neighbors in $X$). \label{claim: many good edges for X}
        \item\label{claim:manygood3} Suppose $|W_2|=\cdots=|W_r|$. For every $v \in V_1$ there are at most $(1+\epsilon) \prod_{i=2}^r|W_i|
        \,p^{\binom{r-1}{2}}$ hyperedges in $\tilde{F}$ with $v$ in its common neighborhood.
    \end{enumerate}
\end{claim}

\begin{proof}
    For part~\ref{claim: many copies}, we will use Janson's inequality (Lemma~\ref{lem:janson}). 
    The set $[N]$ is the edge set of $V_2\sqcup \ldots \sqcup V_r$, with each edge chosen independently with probability $p$. 
    Let $\cI=\{i: e_i \in F[W] \}$, where $F[W]  =F_{V_1}[W_2, \ldots, W_r]$.
    The set $D_i$ corresponds to the set of ${r-1\choose 2}$ edges of the hyperedge $e_i\in F[W]$. 
    For every $D_i$, the random variable $I_i$ is the indicator of the event that all $\binom{r-1}{2}$ edges of $D_i$ are chosen. 
    Let $S= \sum_{i\in \cI}I_i$.
    Then, $\mu =\E[S]=\bigl\|F[W]\bigr\|\, p^{\binom{r-1}{2}}$. We compute the quantity $\Delta$. 
    Let $e_i$ denote the hyperedge corresponding to the ${r-1\choose 2}$ edges defined by $D_i$.
    \begin{align*}
        \Delta = \sum_{(i,j), i\sim j, i\neq j}\E[I_i I_j]&=\sum_{i\in\cI}\sum_{\substack{j:2\leq |e_i\cap e_j|\leq r-2}} p^{2\binom{r-1}{2} - \binom{|e_i\cap e_j|}{2}}\\
        &\leq \|F[W] \| p^{2\binom{r-1}{2}} \sum_{k=2}^{r-2} n^{r-1-k}p^{ - \binom{k}{2}}\\
        &= \mu^2\sum_{k=2}^{r-2} \frac{n^{r-1-k}}{\|F[W]\|}\cdot p^{ - \binom{k}{2}}\\
    \end{align*}

By Lemma~\ref{lem: auxiliary F graph}~\ref{high min degree}, 
\begin{align}
    \bigl\|F[W]\bigr\| &\geq \prod_{i=2}^{r} |W_i| - n\cdot(r-2)\epsilon n^{r-2} \nonumber \\
    &\geq \biggl(1- \frac{(r-2)\epsilon}{\delta^{r-1}}\biggr)\prod_{i=2}^{r} |W_i|
    \geq \bigl(\delta^{r-1}- (r-2)\epsilon\bigr)n^{r-1}. \label{eqn:aux Fw} 
\end{align} 
Hence using~\eqref{eqn:aux Fw} and $p\geq Cn^{-2/(r-1)}$ and $C>1$,
\begin{align*}
        \Delta &\leq  \mu^2\cdot  \sum_{k=2}^{r-2}\biggl(\frac{n^{-k}}{\delta^{r-1}-(r-2)\epsilon}\cdot p^{-\binom{k}{2}}\biggr)\\
        &\leq \frac{C^{-1}}{\delta^{r-1}-(r-2)\epsilon}\cdot\mu^2 \sum_{k=2}^{r-2}n^{-k+\frac{k(k-1)}{r-1}}\\
        &\leq \frac{C^{-1}(r-3)}{\delta^{r-1}-(r-2)\epsilon}\cdot\mu^2 n^{-2+\frac{2}{r-1}}.
\end{align*}

Since $\epsilon \ll \delta\ll 1/r$, from \eqref{eqn:aux Fw} we obtain 
\begin{align*}
    \mu^{-1}\leq (\delta^{r-1}-(r-2)\epsilon)^{-1} C^{-{r-1 \choose 2}} n^{-1} &\leq (1/4) \delta^{-r} C^{-1}n^{-1}, \\
 \Delta\mu^{-2} \leq (r-3) (\delta^{r-1} - (r-2)\epsilon)^{-1}C^{-1}n^{-1} &\leq (1/4)\delta^{-r}C^{-1}n^{-1}.   
\end{align*}
Hence, by Janson's inequality, Lemma~\ref{lem:janson}, as long as $C$ is sufficiently large, i.e., $1/C\ll \epsilon$, we have that
\begin{align*}
    \prob(S \leq \mu-\epsilon\mu ) &\leq \exp\biggl\{ -\frac{\epsilon^2 \mu^2}{2(\mu + \Delta)} \biggr\}=\exp\biggl\{ -\frac{\epsilon^2 }{2\bigl(\mu^{-1} + \Delta\mu^{-2}\bigr)} \biggr\} \\
    &\leq \exp\bigl\{ -\epsilon^2\delta^r\cdot Cn\bigr\}\leq \exp\bigl\{-rn \bigr\}.
\end{align*}

Part~\ref{claim:manygood1} follows by the union bound as there are at most $(2^{n})^{r-1}=\exp\{(r-1)n\ln 2\}$ choices for the sets $ W_2, \ldots, W_{r}$. Therefore, \whp~for any choice of $W_2, \ldots, W_r$ the number of hyperedges in $\tilde{F}\cap F[W]$ is at least 
\begin{align*}
    (1-\epsilon)\bigl\|F[W]\bigr\|\, p^{r-1\choose 2} &\geq (1-\epsilon)\biggl(1- \frac{(r-2)\epsilon}{\delta^{r-1}}\biggr)\prod_{i=2}^{r} |W_i|\, p^{r-1\choose 2} \\
    &\geq 
    \Bigl(1-\frac{r\epsilon}{\delta^{r-1}}\Bigr)\prod_{i=2}^{r} |W_i|\, p^{\binom{r-1}{2}} , 
\end{align*}
using $\eqref{eqn:aux Fw}$. This concludes the proof of part~\ref{claim:manygood1}.

For part~\ref{claim:manygood2}, we again use Janson's inequality (Lemma~\ref{lem:janson}), this time to count the number of edges in $\tilde{F}\cap F_{X}[W_1, \ldots, W_r]$.
The set $[N]$ is again the edge set of $V_2\sqcup \ldots \sqcup V_r$, with each edge chosen independently with probability $p$.
Let $\cI=\{i: e_i \in F_X[W] \}$, where $F_X[W]  =F_{X}[W_2, \ldots, W_r]$.
The set $D_i$ corresponds to the set of ${r-1\choose 2}$ edges of the hyperedge $e_i\in F_{X}[W]$.
For every $D_i$, the random variable $I_i$ is the indicator of the event that all $\binom{r-1}{2}$ edges of $D_i$ are chosen.
Then, $\mu = \E[S]=\|F_{X}[W]\|\, p^{{r-1 \choose 2}}$.
We compute the quantity $\Delta$. 
Let $e_i$ denote the hyperedge corresponding to the ${r-1\choose 2}$ edges defined by $D_i$.

 A calculation identical to the one in part~\ref{claim:manygood1} gives
\begin{align*}
        \Delta
        &\leq  \mu^2\sum_{k=2}^{r-2} \frac{n^{r-1-k}}{\|F_{X}[W]\|}\cdot p^{ - \binom{k}{2}}
\end{align*}

By Lemma~\ref{lem: auxiliary F graph}~\ref{high min deg F_X},
\begin{align*}
    \bigl\|F_X[W]\bigr\| 
   &\geq \biggl(1- \frac{2(r-2)\epsilon}{\delta^{r-1}}\biggr)\prod_{i=2}^{r} |W_i|
    \geq \bigl(\delta^{r-1}- 2(r-2)\epsilon\bigr)n^{r-1}. 
\end{align*}

Hence, with $p\geq Cn^{r-1}$ and $C>1$,

\begin{align*}
        \Delta 
        &\leq \mu^2\cdot  \sum_{k=2}^{r-2}\biggl(\frac{n^{-k}}{\delta^{r-1} -2(r-2)\epsilon}\cdot p^{-\binom{k}{2}}\biggr).
\end{align*}

Since $\epsilon \ll \delta \ll 1/r$, and $1/C\ll \epsilon$, Janson's inequality again gives
\begin{align*}
    \mu^{-1}, \Delta\mu^{-2}  &\leq (1/4) \delta^{-r} C^{-1}n^{-1} \qquad\mbox{ and }\\
    \prob(S \leq \mu-\epsilon\mu ) 
    &\leq \exp\bigl\{-rn \bigr\}.
\end{align*}

Part~\ref{claim:manygood2} follows by the union bound as there are at most $(2^{n})^r=\exp\{rn\ln 2\}$ choices for the sets $X, W_2, \ldots, W_{r}$. Therefore, \whp~for any choice of $X,W_2, \ldots, W_r$ the number of hyperedges in $\tilde{F}\cap F_{X}[W]$ is at least 
\begin{align*}
    (1-\epsilon)\bigl\|F_X[W]\bigr\|\, p^{r-1\choose 2}
    &\geq 
    \Bigl(1-\frac{2r\epsilon}{\delta^{r-1}}\Bigr)\prod_{i=2}^{r} |W_i|\, p^{\binom{r-1}{2}} . 
\end{align*}
This concludes the proof of part~\ref{claim:manygood2}.

For part~\ref{claim:manygood3}, consider any $v \in V_1$ and let $Z_i$ denote the event that the hyperedge $e_i \in F$,  appears in $\tilde{F}$, and $v$ is incident to all vertices of $e_i$ for each $i\in \{1,\ldots, e_F \}$. Let $Z = \sum_{i=1}^{e_F}Z_i$. Note that $\E[Z] \leq \prod_{i=2}^{r}|N_{V_i}(v_1)|p^{\binom{r-1}{2}} \leq n^{r-1}p^{\binom{r-1}{2}}$.

Lemma~\ref{lem:upper tail cliques} is due to Demarco and Kahn~\cite{DemarcoKahn}. Although Lemma~\ref{lem:upper tail cliques} is stated for $G(n,p)$ as Theorem 2.3 in~\cite{DemarcoKahn}, their proof actually gives the multipartite setting as described below.
\begin{lemma}[Demarco and Kahn~\cite{DemarcoKahn}, Theorem~2.3]\label{lem:upper tail cliques}
    Let $r\geq 2$ and $\epsilon>0$. Let $Z$ denote the number of copies of $K_{r-1}$ in $G_{r-1}(n,p)$. If $p\geq n^{-2/(r-2)}$, then
    \begin{align*}
        \prob\bigl(Z > (1+\epsilon)\E[Z]\bigr)< \exp{ \Bigl\{ -\Omega_{\epsilon,r} \Bigl( \min \Bigl\{n^2p^{r-2} \log(1/p), n^{r-1}}p^{\binom{r-1}{2}} \Bigr\}\Bigr)\Bigr\}.
    \end{align*}
\end{lemma}
Recalling that $p \geq Cn^{-2/(r-1)}$, Lemma~\ref{lem:upper tail cliques} applies. 
Substituting $p\geq Cn^{-2/(r-1)}$ into the exponent gives
\begin{align*}
    &\prob\bigl(Z > (1+\epsilon)\E[Z]\bigr) \\
    &\hspace{24pt} < \exp\biggl\{ -\Omega_{\epsilon,r} \Bigl( \min \Bigl\{n^2\bigl(Cn^{-\frac{2}{r-1}}\bigr)^{r-2} \log\bigl(C^{-1}n^{\frac{2}{r-1}}\bigr), n^{r-1}\bigl(Cn^{-\frac{2}{r-1}}\bigr)^{\binom{r-1}{2}} \Bigr\}\Bigr)\biggr\} \\
    &\hspace{24pt}\leq \exp\biggl\{-\Omega_{\epsilon,r}\Bigl(\min\Bigl\{n^{2/(r-1)}\log n,n\Bigr\}\Bigr)\biggr\}.
\end{align*}
By taking the union bound over all vertices $v\in V_1$, and observing that $\E[Z]\leq n^{r-1}p^{\binom{r-1}{2}}$, we obtain that the probability that any $v\in V_1$ has more than $(1+\epsilon)n^{r-1}p^{\binom{r-1}{2}}\geq (1+\epsilon)\E[Z]$ hyperedges of $\tilde{F}$ in its neighborhood is at most
\begin{align*}
    n\cdot\exp\biggl\{-\Omega_{\epsilon,r}\Bigl(\min\Bigl\{n^{2/(r-1)}\log n,n\Bigr\}\Bigr)\biggr\}
\end{align*}
which goes to zero and part~\ref{claim:manygood3} follows. This concludes the proof of Claim~\ref{claim:many good}.
\end{proof}

We use a random greedy process to choose a matching $M$ of size $(1-\delta)n$ in $\tilde{F}$. 
That is, having chosen vertex-disjoint edges $e_1, \dots, e_k \in \tilde{F}$ with $k < (1-\delta)n$, 
we choose $e_{k+1}$ uniformly at random from the set of all edges that do not share at least one vertex with $e_1,\dots, e_k$. With $W_i = V_i - \cup_{j=1}^{k} V(e_j)$,  Claim~\ref{claim:many good}~\ref{claim:manygood1} gives that the number of edges available for $e_{k+1}$ is at least 
\begin{align*}
    (1 -\epsilon)\bigl\|F_{V_1}[W_2, \dots, W_r]\bigr\|p^{r-1\choose 2} 
    &\geq \Bigl(1-\frac{r\epsilon}{\delta^{r-1}}\Bigr)(\delta n)^{r-1}p^{r-1\choose 2} \\
    &\geq \Bigl(1-\frac{r\epsilon}{\delta^{r-1}}\Bigr)(\delta n)^{r-1}C^{r-1\choose 2} n^{-(r-2)}= \Omega(n) .
\end{align*}

Now we will show that the matching obtained has the property that $B=B(V_1,M)$ is $\bigl(\epsilon',(d/2)^{r-1}/4\bigr)$-super-regular.

We first check that the minimum degree condition, Definition~\ref{def: super regular}~\ref{def: sr degree} is satisfied. 
Indeed, for any $e\in M$, the definition of $F=F_{V_1} [V_2,\ldots,V_r]$ (Definition~\ref{def:auxgraph}) gives that $|N_{B}(e)| \geq (d/2)^{r-1}|V_1| > \bigl((d/2)^{r-1}/4\bigr)|V_1|$. 

Next, consider an arbitrary $v\in V_1$ and recall we choose $n=|V_1|=\cdots=|V_r|$.
Consider the first $(d/2)n$ hyperedges added to $M$ by the greedy process. Let $W_{i}$ be the subset of $N_{V_i}(v)$ that does not intersect any of the previously chosen edges $e_{1}, \dots, e_{dn/2}$. Since $v$ has at least $(d/2)n \geq \delta n$ neighbors in each of $V_2, \dots, V_r$ then by Claim~\ref{claim:many good}~\ref{claim:manygood1}, 
\begin{align*}
    (1 -\epsilon)\bigl\|F_{V_1}[W_2, \dots, W_r]\bigr\|p^{\binom{r-1}{2}}\geq\biggl(1-\frac{r\epsilon}{\delta^{r-1}}\biggr)\biggl(\frac{d}{2}n\biggr)^{r-1}p^{\binom{r-1}{2}}.
\end{align*}

On the other hand by Claim~\ref{claim:many good}~\ref{claim:manygood3}, there are at most $(1+\epsilon)n^{r-1}p^{\binom{r-1}{2}}$ hyperedges of $F_{V_1}[W_2,\ldots,W_r]$ with $v$ being a common neighbor. Therefore, 
\begin{align*}
    \prob\bigl(e_j\in N_{B}(v) | e_1, \dots, e_{j-1}\bigr) \geq \frac{\bigl(1-\frac{r\epsilon}{\delta^{r-1}}\bigr)\bigl(dn/2\bigr)^{r-1}p^{\binom{r-1}{2}}}{(1+\epsilon)n^{r-1}p^{\binom{r-1}{2}}}\geq \frac{d^{r-1}}{2^r},
\end{align*}
as long as $\epsilon \ll \delta$.
As observed in~\cite{bottcher2023triangles}, this holds independently of the process, as such the process dominates a binomial distribution with parameters $(d/2)n$ and $d^{r-1}/2^r$. 
That is, the probability of the event that $j$ members of $\bigl\{e_1,\ldots,e_{dn/2}\bigr\}$ are in $N_B(v)$ is at least the probability of that same event in ${\rm Bin}\bigl(dn/2,d^{r-1}/2^r\bigr)$.

Let $A$ be distributed according to ${\rm Bin}\bigl(dn/2,d^{r-1}/2^r\bigr)$, then the probability that there are at least $k$ members of $\{e_1, \ldots, e_{dn/2} \}$ in $N_B(v)$ is at most the probability that $\prob(A \leq k)$. 
We apply the Chernoff bound, Lemma~\ref{lemma: chernoff}, to $A$,  
$$ \prob\bigl(A \leq (1-\xi) \E[A]\bigr) \leq 2\exp\biggl(-\frac{\xi^3}{3}\E[A]\biggr).$$
Setting $\xi=1/2$, we obtain that the probability that $A$ is smaller than $(1/2)\E[A]=\frac{d^{r-1}}{2^{r+1}}n$ is at most $\exp\{-\Omega(n)\}$. By the union bound, \whp~for all $v\in V_1$, we have that 
$|N_{B}(v)| \geq \frac{d^{r-1}}{2^{r+1}}n\geq \frac{d^{r-1}}{2^{r+1}}|M|$. This verifies the minimum degree condition, Definition~\ref{def: super regular}~\ref{def: sr degree}.

We next move on to showing the regularity, Definition~\ref{def: super regular}~\ref{def: sr density}, but first we'll need to prove Claim~\ref{claim: large sets, good edges} below which states that, \whp~there does not exist too many hyperedges that are not in $F_{X}[V_2, \ldots, V_r]$.
\begin{claim}\label{claim: large sets, good edges}
    \Whp~for all $X\subseteq V_1$ with $|X| =\beta n$, there are at most $\gamma n$ edges in $M$ that are not in $F_X=F_{X}[V_2, \ldots V_r]$.
\end{claim}
\begin{proof}[Proof of Claim~\ref{claim: large sets, good edges}]
    Let $X \subseteq V_1$ be given and let $k=|M|=(1-\delta)n$. 
    For each $j\in \{1,\ldots,k-1\}$, we apply Claim~\ref{claim:many good}~\ref{claim:manygood3} to see that are at most $(1+\epsilon)(n - j)^{r-1} p^{\binom{r-1}{2}}$ hyperedges in the subgraph of $\tilde{F}$ induced by the vertices of $\sqcup_{i=2}^{r}V_i - \cup_{\ell=1}^{j-1}V(e_\ell)$, which are the vertices that are available for selecting $e_{j+1}$. 
    By Claim~\ref{claim:many good}~\ref{claim: many good edges for X}, at least $(1-\frac{2r\epsilon}{\delta^{r-1}})( n-j)^{r-1}p^{\binom{r-1}{2}}$ hyperedges are good for $X$. Therefore,
    \begin{align*}
        \prob\bigl(e_{j+1} \text{ not in } F_X \mid e_1, \ldots, e_{j}\bigr) &\leq 1- \cfrac{\bigl(1-\frac{2r\epsilon}{\delta^{r-1}}\bigr)(n-j)^{r-1}p^{r-1\choose 2}}{(1+\epsilon)(n-j)^{r-1}p^{r-1\choose 2}}\\
        &\leq 1- \cfrac{1- \frac{2r\epsilon}{\delta^{r-1}}}{1+\epsilon}
        \leq \frac{3r\epsilon}{\delta^{r-1}}.
    \end{align*}
    As long as as $\epsilon\ll \delta <1$.
    As before, this holds independently of the history of the process. This process dominates a binomial distribution ${\rm Bin}\bigl((1-\delta)n,3r\epsilon/\delta^{r-1}\bigr)$,  with mean $(1-\delta)3r\epsilon n/\delta^{r-1} \geq 2r\epsilon n/\delta^{r-1}$. If $\overline{F_X}$ are the edges not in $F_X$, then with $\gamma n  \geq 7\cdot 2r\epsilon n/\delta^{r-1}$, we apply the Chernoff bound (the ``Moreover'' statement of Lemma~\ref{lemma: chernoff}) to obtain
    \begin{align*}
        \prob\bigl(|\overline{F_X}| > \gamma n \bigr) \leq \exp\bigl\{- \gamma n\bigr\}.
    \end{align*}
    The number of choices for $X$ is at most
    \begin{align*}
        \binom{n}{\beta n} \leq \biggl(\frac{en}{\beta n}\biggr)^{\beta n} \leq \exp\bigl\{\bigl(\beta+\beta \ln(1/\beta)\bigr)\,n\bigr\}\leq \exp\bigl\{\gamma n /2\bigr\}
    \end{align*}
    which results from choosing $\beta+\beta \ln(1/\beta) \leq \gamma/2$.
    By the union bound, \whp~there is no $X$ for which $X$ has at least $\gamma n$ edges not in $F_X$. This concludes the proof of Claim~\ref{claim: large sets, good edges}. 
\end{proof}

Addressing the regularity condition, let $X' \subseteq V_1$ and $M' \subseteq M$ with $|X'| \geq \epsilon' n$ and $|M'| \geq \epsilon' |M|\geq\epsilon'(1-\delta)n$.
We arbitrarily partition $X'$ into sets of size $\beta n$ and apply Claim~\ref{claim: large sets, good edges} to each part of the partition of $X'$ and recall that every hyperedge of $F$ has at least $(d/2)^{r-1}\beta n$ common neighbors in each part of $X'$. We obtain:
\begin{align*}
    e_{B}(M',X') &\geq \bigl(|M'| - \gamma n\bigr) \biggl\lfloor\frac{|X'|}{\beta n}\biggr\rfloor (d/2)^{r-1}\beta n  \\
    &\geq (d/2)^{r-1} |M'||X'|\biggl(1 - \frac{\gamma n}{|M'|}\biggr) \biggl(1-\frac{\beta n}{|X'|}\biggr)  \\
    &\geq (d/2)^{r-1}|M'| |X'|(1/4)
\end{align*}
where we used that $\delta \leq 1/2$, $\beta\leq\epsilon'/2$, and $\gamma \leq \epsilon'/4$. This establishes the regularity condition, Definition~\ref{def: super regular}~\ref{def: sr density} and concludes the proof of Lemma~\ref{lem: eps-reg matching}. \qed




\section{A note on minimum degree}\label{section:note on min degree}
In \cite{balogh2019tilings}, Balogh, Treglown, and Wagner ask whether the $\alpha n$ in Theorem~\ref{thm:btw} can be replaced with a sublinear term.
In the work of Chang, Han, Kohayakawa, Morris, and Mota~\cite{hypergraphfactors2022}, this question is answered in the negative.
That is, for $\omega =o(n)$, they provide a graph $G$ on $n$ vertices with $\delta(G) \geq n/\omega$ and they show that the threshold for a perfect matching in $G\cup G(n,p)$ is at least $n^{-1} \log\omega$.

Here we adapt the construction of \cite{hypergraphfactors2022} to the multipartite setting.
That is, we give a balanced $r$-partite graph $G$ on $rn$ vertices with $\delta^*(G)\geq n/\omega$ and show that the threshold for a perfect $K_r$-tiling in $G':=G\cup G_r(n,p)$ is at least $n^{-2/r}\ln^{1/\binom{r}{2}}\omega$.
The proof, however, requires more machinery than in~\cite{hypergraphfactors2022}. 
We will use both the lower and upper bounds of Corollary~\ref{cor:janson} as well as Chebyshev's inequality. 

For the construction, let $r\geq 3$ be an integer, and set $\eta = \frac{1}{3r \omega}$. 
Let $p= n^{-2/r} \ln^{1/{r\choose 2}}\omega$ where $\omega = o(n)$.
Consider the $r$-partite graph $G= (V_1 \sqcup \ldots \sqcup V_r;E)$ with $V_i = A_{i} \sqcup B_{i}$ 
with $|A_i| = \eta n$ and $|B_i|=(1-\eta)n$, for $i\in[r]$.
The graph has all edges between $(A_i, A_j)$ and all edges between $(A_i, B_j)$, for all $i\neq j \in [r]$.
Consequently, $\delta^{*}(G) \geq \eta n$. 

Our present goal is to compute the probability that a fixed vertex $v\in B_1$ is not in any $K_r$ in $B:=\sqcup_{i=1}^{r} B_i$.
Note that $G'[B]$ is distributed according to $G_r((1-\eta)n,p)$.
We say that a vertex $v$ is \emph{isolated} if there is no copy of $K_r$ in $B$ that contains it. 
In other words, $v$ is isolated in the $r$-uniform hypergraph formed by copies of $K_r$.

Let $\cI$ be the set of $K_r$'s that contain $v$ and for each $i\in\cI$, let $I_i$ be the indicator variable for the event that the $i^{\rm th}$ copy of $K_r$ to which $v$ belongs has all of its ${r\choose 2}$ edges in $G'[B]$.
It follows that $\prob(I_i=1)=p^{r \choose 2}$ for all $i\in\cI$. 
Let $S= \sum_{i\in \mathcal I} I_i$. In our context $\mu' =  -|\cI|\ln\bigl(1-p^{r \choose 2}\bigr) =-n^{r-1}\ln\bigl(1-p^{r \choose 2}\bigr) $.

By Corollary~\ref{cor:janson}, using $\ln (1-x)\geq -x-x^2$ for all $x\in[0,1/2)$
\begin{align*}
    \prob(S=0) &\geq \exp\{-\mu'\}=\exp\Bigl\{n^{r-1}\ln\Bigl(1-p^{r \choose 2} \Bigr)\Bigr\}\\
    &\geq \exp\Bigl\{-n^{r-1}p^{r\choose 2}-n^{r-1}p^{2{r\choose 2}}\Bigr\} 
    =\frac{1}{\omega}\exp\bigl\{-(\ln \omega)^2/n^{r-1}\bigr\}.
\end{align*}

Let $X$ be the number of vertices of $B_1$ that do not belong to any clique of $G_r((1-\eta)n,p)$. Then, $\E[X] \geq  (n/\omega) \exp\bigl\{-(\ln \omega)^2/n^{r-1}\bigr\}$.

In order to apply Chebyshev's inequality (Lemma~\ref{lem:chebyshev}), we compute $\var(X) = \E[X^2] - \E[X]^2$.
For all $v\in B_1$, let $X_v$ be the indicator that $v$ is isolated.
To that end, 
\begin{align}
    \E[X^2] = \sum_u\E[X_u] + \sum_{u}\sum_{v\neq u}\E[X_uX_v] = \E[X] + \sum_{u}\sum_{v\neq u}\E[X_uX_v].\label{eq:var}
\end{align}
and we will again use Corollary~\ref{cor:janson} (this time, the upper bound) in order to estimate $\E[X_uX_v]$. 

Let $\cI$ denote the set of all $K_{r}$'s in $G'[B]$ that contain either $u$ or $v$.
Then $\mu = 2n^{r-1}p^{r \choose 2}=2\ln{\omega}$.
In order to calculate $\Delta$, we observe that two copies of $K_r$ are adjacent in the dependency graph if and only if they share at least two vertices. 

Thus we first consider pairs of $K_r$'s that contain the same $B_1$ vertex (either $u$ or $v$) and then consider pairs where one contains $u$ and the other contains $v$. 
\begin{align*}
    \Delta &= \sum_{(i,j): i \sim j,i\neq j} \E[I_i I_j] \\
    &= 2\sum_{i\in\cI} \sum_{\ell =1}^{r-2} \binom{r-1}{\ell} n^{r-\ell-1}p^{2{r\choose 2} - {\ell+1 \choose 2}} + \sum_{i\in\cI}\sum_{\ell=2}^{r-1}\binom{r-1}{\ell}n^{r-\ell-1}p^{2{r\choose 2} - {\ell \choose 2}} \\
    &= \sum_{\ell =1}^{r-2} 2\binom{r-1}{\ell} n^{2r-\ell-2}p^{2{r\choose 2} - {\ell+1 \choose 2}} + \sum_{\ell=2}^{r-1}\binom{r-1}{\ell}n^{2r-\ell-2}p^{2{r\choose 2} - {\ell \choose 2}} \\ 
    &= n^{2r-3+2/r}p^{2{r\choose 2}}\Biggl[\sum_{\ell =1}^{r-2} 2\binom{r-1}{\ell} n^{1-2/r-\ell}p^{-{\ell+1 \choose 2}} + \sum_{\ell=2}^{r-1}\binom{r-1}{\ell}n^{1-2/r-\ell}p^{- {\ell \choose 2}}\Biggr].
\end{align*}
Observe for our choice of $p$, $n^{1-2/r-\ell}p^{-{\ell+1 \choose 2}} \leq 1$ for all $\ell\in\{1,\ldots, r-2\}$. For $r\geq 3$ and $n$ sufficiently large,
\begin{align*}
    \Delta \leq n^{2r-3+2/r}p^{2{r\choose 2}} \bigl[ 2\cdot 2^{r-1} + n^{-1}2^{r-1}\bigr]
    \leq n^{-1/4}.
\end{align*}

Returning to \eqref{eq:var}, Corollary~\ref{cor:janson} gives,
\begin{align*}
    \var(X)=\E[X^2]-(\E[X])^2 \leq \E[X] + n(n-1) \exp\{-\mu + \Delta \}-(\E[X])^2 
\end{align*}

We use Lemma~\ref{lem:chebyshev} to show that $X > \E[X]/2$ \whp. Using the fact that $\mu=2\ln\omega$, 
\begin{align*}
    \prob\bigl(|X-\E[X]| > \E[X]/2\bigr) 
    &\leq \frac{4}{(\E[X])^2} \Bigl(\E[X] + n(n-1) \exp\{-\mu + \Delta \}-(\E[X])^2\Bigr) \\
    &\leq \frac{4}{(\E[X])^2} n(n-1) \exp\{-2\ln\omega + \Delta \} + \frac{4}{\E[X]}-4   
\end{align*}

Now using the inequality $\E[X] \geq  (n/\omega) \exp\bigl\{-(\ln \omega)^2/n^{r-1}\bigr\}$ and the fact that $e^x\leq 1+2x$ for all $x<1$ and $n$ is sufficiently large to obtain
\begin{align*}
    \prob\bigl(|X-\E[X]| > \E[X]/2\bigr)
    &\leq 4\exp\biggl\{\Delta+\frac{2(\ln \omega)^2}{n^{r-1}}\biggr\} + 4\biggl(\frac{\omega}{n}\biggr) \exp\biggl\{\frac{(\ln \omega)^2}{n^{r-1}}\biggr\}-4\\ 
    &\leq 2\Delta+\frac{4(\ln \omega)^2}{n^{r-1}} + 4\biggl(\frac{\omega}{n}\biggr) \exp\biggl\{\frac{(\ln \omega)^2}{n^{r-1}}\biggr\} 
\end{align*}
which goes to $0$ because $\Delta\leq n^{-1/4}$.

So with high probability, for each $i\in[r]$, $B_i$ contains at least $\E[X]/2\geq \frac{n}{2\omega}\exp\bigl\{-\frac{(\ln \omega)^2}{n^{r-1}}\bigr\}\geq\frac{n}{3\omega}$ isolated vertices.
The total number of isolated vertices is therefore at least $\frac{rn}{3\omega}$. 
However, if there were a perfect $K_r$-tiling, there must be at least $\frac{1}{r-1}\cdot\frac{rn}{3\omega}$ vertices in $A=\sqcup_{i=1}^{r}A_i$ in order to cover all of the isolated vertices. 

But $|A|=r\eta n=\frac{n}{3\omega}< \frac{1}{r-1}\cdot\frac{rn}{3\omega}$, a contradiction.
Thus this graph with minimum degree $\delta^*\geq n/(3r\omega)$ will not have a perfect $K_r$-tiling when perturbed with probability $p=n^{-2/r}\ln^{1/{r\choose 2}}\omega$. 

Consequently, the threshold for a randomly perturbed balanced multipartite graph with sublinear minimum degree requires a polylog factor for a perfect $K_r$-tiling.


\section{Concluding Remarks}\label{section:conclusion}
A recent work of Han, Morris, and Treglown~\cite{HMT} studies the problem of determining the correct probability threshold in the case of $\alpha$ large in the usual (non-multipartite) setting. Their main result captures a ``jumping'' phenomenon in the threshold. Their main result is Theorem~\ref{thm:HMT bridging} below.

Let $\cG(\alpha)$ denote the class of graphs on $n$ vertices with minimum degree at least $\alpha n$. 

\begin{theorem}[Han, Morris, and Treglown~\cite{HMT}, Theorem 1.3]\label{thm:HMT bridging}
Let $2\leq k \leq r$. Then given $1-\frac{k}{r} < \alpha < 1 - \frac{k-1}{r}$, then the threshold for the property of having a perfect $K_r$-tiling in $G \cup G(n,p)$, where $G\in\cG(\alpha)$, is $n^{-2/k}$.
\end{theorem}

In particular, Theorem~\ref{thm:HMT bridging} shows that if $\delta(G) \geq n/3$, then the probability threshold for a $K_3$-tiling in $G\cup G(n,p)$ increases from $n^{-2/3}$ to $n^{-1}$. Later, B\"{o}ttcher, Parczyk, Sgueglia, and Skokan showed that when $\delta(G) = n/3$, that the probability threshold for a $K_3$-tiling in $G\cup G(n,p)$ is $n^{-1}\log{n}$.
\begin{theorem}[B\"{o}ttcher, Parczyk, Sgueglia, and Skokan \cite{bottcher2023triangles}, Theorem 1.3]
The threshold for a perfect $K_3$-tiling in $G\cup G(n,p)$, where $G\in \cG(1/3)$ is $n^{-1}\log{n}$.
\end{theorem}
This leads to a natural question in the randomly perturbed tripartite case. We leave this as an open question.

\begin{problem}
    Determine the threshold for a perfect $K_3$-tiling in the tripartite graph $G_{n}\cup G_3(n,p)$ for sequences $(G_n)\subseteq\cG_r(\alpha;n)$ for all $\alpha\in [1/3,2/3)$.
\end{problem}

We also pose an extension of Theorem~\ref{thm: main} in the case in which $K_r$ is replaced by any graph $H$ with chromatic number equal to the number of parts in the multipartition.  
\begin{problem}
    Fix a graph with chromatic number $r\geq3$ and $\alpha \leq 1/|V(H)|$. Determine the threshold for an $H$-tiling of size $\lfloor n/|V(H)|\rfloor$ in $G_{n}\cup G_r(n,p)$ for sequences $(G_n)\subseteq\cG_r(\alpha;n)$. 
\end{problem}
\section{Acknowledgments}

Gomez-Leos' research is partially supported by National Science Foundation RTG grant DMS-1839918. Martin's research is partially supported by Simons Foundation Collaboration Grant for Mathematicians \#709641. 

The authors offer their appreciation to Sahar Diskin for pointing us to~\cite{GerkeMcDowell}. They also wish to thank Sam Spiro for useful conversations. 

\bibliographystyle{abbrv}
\bibliography{bibfile}
\end{document}